\documentclass[10pt]{amsart}
\pdfoutput=1
\theoremstyle{plain}
\usepackage{amsmath, amsfonts, amsthm, amssymb,amscd,bbold}
\usepackage{graphics}
\usepackage{color}
\usepackage{epsfig}
\usepackage{tikz}
\usetikzlibrary{arrows}
\usepackage[all]{xy}
\usepackage{todonotes}
\theoremstyle{plain}
\newtheorem{theorem}{Theorem}
\newtheorem{lemma}{Lemma}
\newtheorem{proposition}{Proposition}
\newtheorem{corollary}{Corollary}

\newtheorem{question}{Question}
\newtheorem{claim}{Claim}

\newtheorem*{theoremSnAction}{Theorem~\ref{thm:Sn}}
\theoremstyle{definition}

\newtheorem{definition}{Definition}
\newtheorem{convention}{Convention}
\newtheorem*{sclaim}{Claim}
\newtheorem{remark}{Remark}

\newcommand{\N}{\ensuremath{\mathbb{N}}}

\newcommand{\R}{\ensuremath{\mathbb{R}}}
\newcommand{\Z}{\ensuremath{\mathbb{Z}}}

\newcommand{\C}{\ensuremath{\mathbb{C}}}

\newcommand{\bL}{\ensuremath{\mathbb{L}}}
\newcommand{\Id}{\ensuremath{\mbox{\textbb{1}}}}

\newcommand{\OO}{\ensuremath{\mathbb{O}}}
\newcommand{\XX}{\ensuremath{\mathbb{X}}}

\newcommand{\cP}{\ensuremath{\mathcal{P}}}
\newcommand{\cI}{\ensuremath{\mathcal{I}}}

\newcommand{\cF}{\ensuremath{\mathcal{F}}}

\newcommand{\cA}{\ensuremath{\mathcal{A}}}

\newcommand{\Kh}{\ensuremath{\mbox{Kh}}}
\newcommand{\SKh}{\ensuremath{\mbox{SKh}}}
\newcommand{\CKh}{\ensuremath{\mbox{CKh}}}

\newcommand{\sltwo}{\ensuremath{\mathfrak{sl}_2}}
\newcommand{\slk}{\ensuremath{\mathfrak{sl}_k}}
\newcommand{\glk}{\ensuremath{\mathfrak{gl}_k}}

\newcommand{\comb}{\mathcal{C}omb^4(A)}
\newcommand{\cob}{\mathcal{C}ob^4_{/i}(A)}
\newcommand{\rep}{\mbox{gRep}(\sltwo)}

\newcommand{\coba}{\mathcal{C}ob^3(A)}
\newcommand{\cobal}{\mathcal{C}ob^3_{/\ell}(A)}
\newcommand{\komh}{\mbox{Kom}_{/h}}

\newcommand{\mat}{\mbox{Mat}}

\newcommand{\kobha}{\mbox{Kob}_{/h}(A)}
\newcommand{\kobpha}{\mbox{Kob}_{/\pm h}(A)}
\newcommand{\gltwo}{\ensuremath{\mathfrak{gl}_2}}

\newcommand{\oV}{V^*}
\newcommand{\ov}{\overline{v}}
\newcommand{\exsltwo}{\sltwo(\wedge)}
\newcommand{\End}{\mbox{End}}
\newcommand{\fS}{\mathfrak{S}}

\newcommand{\id}{\mbox{id}}

\author{J. Elisenda Grigsby}
\thanks{JEG was partially supported by NSF CAREER award DMS-1151671.}
\address{Boston College; Department of Mathematics; 301 Carney Hall; Chestnut Hill, MA 02467}
\email{grigsbyj@bc.edu}

\author{Anthony M. Licata}
\address{Mathematical Sciences Institute; Australian National University; Canberra, Australia}
\email{anthony.licata@anu.edu.au}

\author{Stephan M. Wehrli}
\thanks{SMW was partially supported by NSF grant DMS-1111680.}
\address{Syracuse University; Department of Mathematics; 215 Carnegie; Syracuse, NY 13244}
\email{smwehrli@syr.edu}

\title{Annular Khovanov homology and knotted Schur-Weyl representations}

\begin{document}
\bibliographystyle{plain}
\maketitle
\begin{abstract}
Let $\bL \subset A \times I$ be a link in a thickened annulus.  We show that its sutured annular Khovanov homology carries an action of $\exsltwo$, the exterior current algebra of $\sltwo$. When $\bL$ is an $m$--framed $n$--cable of a knot $K \subset S^3$, its sutured annular Khovanov homology carries a commuting action of the symmetric group $\fS_n$. One therefore obtains a ``knotted" Schur-Weyl representation that agrees with classical $\sltwo$ Schur-Weyl duality when $K$ is the Seifert-framed unknot.\end{abstract}

\section{Introduction}

Knot homologies, like the quantum knot polynomials they categorify, are intimately connected to the representation theory of Lie algebras and quantum groups.  Khovanov homology, the first and most basic of these homology theories, can be constructed by categorifying a part of the representation theory of $U_q(\sltwo)$.  Roughly speaking, the idea is to lift the Reshetikhin-Turaev graphical calculus of $U_q(\sltwo)$-intertwiners one level on the categorical ladder.

\subsection{Tangle invariants and link homologies from $U_q(\sltwo)$ categorification}

Let $T$ be a tangle in $D^2\times I$ connecting $n$ points in $D^2 \times \{0\}$ to $m$ points in $D^2\times \{1\}$.
The most basic of the Reshitikhin-Turaev tangle invariants assigns to $T$ a
$U_q(\sltwo)$ homomorphism 
$$\psi(T): V_1^{\otimes n} \longrightarrow V_1^{\otimes m},$$
where $V_1$ is the defining two-dimensional representation of the quantum group $U_q(\sltwo)$.

One flavor of categorified tangle invariant replaces the $\C(q)$ vector spaces $V_1^{\otimes n}$ by 
a graded category $\mathcal{C}(n)$ with Grothendieck group $K_0( \mathcal{C}(n)) \cong V_1^{\otimes n}$; the linear map $\psi(T)$ is then upgraded to a functor
$$
	\Psi(T) : \mathcal{C}(n) \longrightarrow \mathcal{C}(m),
$$
with $K_0(\Psi(T)) = \psi(T)$.  
A fascinating aspect of the story is that the categories $\mathcal{C}(n)$ can be chosen from a number of mathematical subjects.  $\mathcal{C}(n)$ could be a category of coherent or constructible sheaves on an algebraic variety, a Fukaya category of a symplectic manifold, a category of modular representations of a finite group, a category of matrix factorizations, or a category of modules over a finite-dimensional algebra.  The choice which is most directly relevant for the current paper is due to Chen-Khovanov \cite{QA0610054} and independently Brundan-Stroppel \cite{BrundanStroppel}, who define finite-dimensional algebras $A_{n}$ and take $\mathcal{C}(n)$ to be the derived category of left $A_{n}$ modules.  The functor valued tangle invariant $\Psi(T)$ is then given by tensoring with a complex of $(A_{n},A_{m})$ bimodules specified by a cube of resolutions of the tangle $T$.\footnote{The construction of this functor-valued tangle invariant is made more explicit in \cite{QA0610054} than in \cite{BrundanStroppel}, whose focus is on a relationship between the algebras $A_n$ and category $\mathcal{O}$. }   

Given an $(n,n)$ tangle $T$, it is natural to study its closure, $\widehat{T}$, as a link in the thickened annulus $A\times I$. To obtain a topological invariant of this closure, one can then take the derived self-tensor product (Hochschild homology) of the $(A_n,A_n)$ bimodule described above. Alternatively, one can consider its sutured annular Khovanov homology $\SKh(\widehat{T})$, defined by Asaeda-Przytycki-Sikora \cite{MR2113902}.\footnote{Asaeda-Przytycki-Sikora  \cite{MR2113902} in fact introduced a version of Khovanov homology for links in thickened oriented surfaces $F \times I$. The annular case $F = A$ was explored further by L. Roberts in \cite{GT07060741}, who 
related it to Heegaard Floer knot homology as in \cite{MR2141852} (see Sec. \ref{sec:Prelim}).  This annular theory has come to be known as {\em sutured annular Khovanov homology} because of a relationship (cf. \cite{GT08071432}, \cite{AnnularLinks}) with Juh{\'a}sz's sutured version of Heegaard Floer homology \cite{MR2253454}. }
Sutured annular Khovanov homology is defined using an explicit chain complex coming from a cube of resolutions, much in the spirit of Khovanov's original definition of a homology theory for links in $S^3$.

In fact, these two invariants are expected to agree. Explicitly, it is conjectured (cf. \cite[Conj. 1.1]{HochHom})\footnote{For fixed $n \in \Z^+$, Chen-Khovanov and Brundan-Stroppel introduce a further grading $\mathcal{C}(n) = \oplus_{k=0}^n \mathcal{C}(n,k)$ on the category $\mathcal{C}$. A more precise version of the conjecture relates the Hochschild homology of the bimodule associated to the category $\mathcal{C}(n,k)$ with a graded summand $\SKh(\bL; -n+2k) \subseteq \SKh(\bL)$. In \cite{HochHom}, the conjecture is proved in the $k=1$ case.} that: $$HH_*(\Psi(T)) \cong \SKh(\widehat{T}).$$

The expectation that sutured annular Khovanov homology arises as Hochschild homology of bimodules from $U_q(\sltwo)$ categorification suggests that the annular homology groups $\SKh(\bL)$ \emph{themselves} should carry rich structure of representation-theoretic interest.  The goal of the present work is to describe some of this structure directly, in down-to-earth terms, without appealing to either Hochschild homology or higher representation theory.


\subsection{Representation theory and sutured annular Khovanov homology}

 The most basic of the representation-theoretic structures enjoyed by $\SKh(\bL)$ is a linear action of $\sltwo$.  We define this $\sltwo$ action directly on the chain level and check that it commutes with the annular boundary maps. We further show that this $\sltwo$ action is diagram-independent, hence an invariant of the underlying annular link. An amusing corollary of this fact is that the sutured Khovanov homology of an annular link is {\em trapezoidal} with respect to the $\sltwo$ weight space grading (Corollary \ref{cor:trap}).  One conceptual explanation for the $\sltwo$ action comes from the conjecture that $\SKh$ can be realized as Hochschild homology of bimodules in $U_q(\sltwo)$ categorification (see Section \ref{sec:catexp} below).
 
It turns out that $\SKh(\bL)$ has somewhat more symmetry than that provided by the $\sltwo$ action.  The Lie algebra $\sltwo$ is the tangent space to the identity of the Lie group $SL_2$, but if we consider the action of $SL_2$ on itself by conjugation, then the quotient stack $SL_2//SL_2$ also has a ``tangent space," which is actually a complex of sheaves.  The fiber of this complex over the identity has the structure of a $\Z$-graded Lie superalgebra, which we refer to in this paper as the exterior current algebra of $\sltwo$, and denote $\exsltwo$.  As a graded vector space, we have
 $$
 	\exsltwo = \sltwo \oplus \sltwo[1],
$$
with the Lie bracket given essentially by the adjoint action of $\sltwo$ on itself (see Section \ref{sec:exsltwo} for a precise definition.)  We prove the following.

\begin{theorem} \label{thm:extcurrsl2} Let $\bL \subset A \times I$ be an annular link. Then the exterior current algebra $\exsltwo$ acts linearly on $\SKh(\bL)$, and the isomorphism class of this representation is an annular link invariant.
\end{theorem}

The proof of this theorem is direct, as we define the action of the generators of $\exsltwo$ at the chain level and check that the defining relations hold up to homotopy.  An interesting point is that the check of relations uses fundamentally the compatibility of the Khovanov differential and the Lee deformation \cite{Lee}, both with each other and with the additional annular grading of the chain complex.  In contrast to the $\sltwo$ action on $\SKh(\bL)$, a more conceptual explanation for the appearance of the exterior current algebra from categorified quantum groups and Hochschild homology is missing at the moment.

Given Theorem \ref{thm:extcurrsl2}, it is reasonable to reformulate annular Khovanov homology as a functor from the category of annular links (with morphisms the annular link cobordisms) to the category of finite-dimensional graded representations of $\exsltwo$.  It follows from this description that the $\exsltwo$ module structure on $\SKh(\bL)$ is an annular link invariant and that annular link cobordisms induce $\exsltwo$-module homomorphisms.  An important special case is when $\bL$ is the cable of a knot, as in this case the $\exsltwo$-module endomorphisms induced by annular link cobordisms have additional structure.  We prove the following result, a more precise version of which is stated in Section \ref{sec:Sn}.

\begin{theorem} \label{thm:Sn} Let $K \subset S^3$ be a knot, and let $\bL = K_{n,nm} \subset A \times I$ denote its $m$--framed $n$--cable. Then $\SKh(\bL)$ carries commuting actions of $\exsltwo$ and of the symmetric group $\fS_n$. 
\end{theorem}

When $K$ is the unknot, and $\bL = K_{n,0}$ is its Seifert-framed $n$--cable, the positive degree part of $\exsltwo$ acts trivially, so the $\exsltwo$ action reduces to an $\sltwo$ action, and the commuting actions of $\sltwo$ and $\fS_n$ then recover the usual Schur-Weyl representation on the $n$th tensor power of the defining representation of $\sltwo$ (cf. Sec. \ref{sec:SchurWeylRep}).  Thus Theorem \ref{thm:Sn} may be viewed as a generalization of the Schur-Weyl representation to arbitrary framed knots, with the $\sltwo$ action in Schur-Weyl duality upgraded to an action of the exterior current algebra.

The topological implications of the exterior current algebra action certainly merit further exploration. We content ourselves here with recalling that the (filtered) annular Khovanov complex is particularly well-suited to studying braid conjugacy classes \cite{BG} and transverse links with respect to the standard tight contact structure on $S^3$ \cite{Plam}. In particular, it distinguishes braid closures from the closures of other tangles \cite{GN} and detects the trivial $n$--braid among all $n$--braids \cite{BG}. Moreover, by imbedding the solid torus in $S^3$ in the standard way, one obtains a spectral sequence from the sutured annular Khovanov homology of $\bL$  to the ordinary Khovanov homology of $\bL$. Although Plamenevskaya's construction predates annular  Khovanov homology, her transverse link invariant \cite{Plam} is a compelling character in the story described here.  Hunt, Keese, and Morrison recently wrote a computer program which computes both the sutured annular Khovanov homology of braid closures  as well as the spectral sequence to Khovanov homology.  A user's guide to that program, along with some example computations, can be found in the companion paper \cite{HKLM}.

\subsection{The $\sltwo$ action on $\SKh(\bL)$ via categorified quantum groups.}\label{sec:catexp}

Conjecturally, $\SKh(\bL)$ can be realized as $HH_*(\Psi(T))$, where $\Psi(T)$ is a complex of bimodules over the 
Chen-Khovanov/Brundan-Stroppel algebras.  This expectation gives rise to one conceptual explanation for the existence of an $\sltwo$ action on $\SKh(\bL)$.  Namely, on the derived category $D^b(\mathcal{C}(n))$, one should be able to define directly an action of the categorified quantum group $\mathcal{U}(\sltwo)$ defined by Lauda in \cite{Lauda}.\footnote{The existence of such an action follows formally from Koszul duality, since there is an explicit categorical action of $U_q(\sltwo)$ on the Koszul dual of $\mathcal{C}(n)$.  It would be desirable to describe the action of the generating 2-morphisms of Lauda's 2-category on $D^b(\mathcal{C}(n))$ directly, though to our knowledge that has not been done yet.} The defining 1-morphisms $E,F$ in Lauda's 2-category are left and right adjoint to one another, and the adjunction 2-morphisms give rise to endomorphisms
$$
	e,f: HH_*(Y) \longrightarrow HH_*(Y),
$$
where $Y$ can be taken to be any complex of $(A_n,A_n)$ bimodules which commutes with the functors $E$ and $F$.  (The endomorphisms $e,f$ are sometimes referred to as Bernstein trace maps \cite{Bernstein}).  The further structure in Lauda's 2-category then implies that
the maps $e,f,h=[e,f]$ will satisfy the defining relations of the Lie algebra $\sltwo$ \cite{Lauda2}.
In particular, if one takes the functor $Y$ to be $\Psi(T)$ for an $(n,n)$ tangle $T$, one should obtain in this way an $\sltwo$ action on $HH_*(\Psi(T))$.  At the moment, it is not clear to us how to use the categorified quantum group to obtain an action of the exterior current algebra directly on $HH_*(\Psi(T))$.  However, we should note that the closely related \emph{polynomial  current algebra} of $\sltwo$ does appear in \cite{Lauda2}.

The existence of an $\sltwo$ action on the annular Khovanov homology of a link also clarifies the relationship between this homology and the skein module of $A\times I$.  Namely, the skein module of a three-manifold $M$ is isomorphic (at least at $q=1$) to the coordinate ring of the $SL_2$-character variety of $\pi_1(M)$ \cite{PS}.  In the case when $M = A\times I$, this description essentially reduces to an identification between the skein module of $A\times I$ and the $W$-invariants in the coordinate ring of $T$, where here $W=S_2$ is the Weyl group of $SL_2$ and $T\subset SL_2$ is the associated maximal torus.  Thus, the skein module of $A \times I$ is isomorphic to $W$-invariant functions on $T$, which in turn may be identified with the Grothendieck group of the category of representations of $\sltwo$.  Thus, from this point of view, the precise relationship between $\SKh$ and the skein module of the annulus naturally involves the representation theory of $\sltwo$.  More precisely, we expect that, given a link $\bL$ in $A\times I$, the class of $\SKh(\bL)$ in the Grothendieck group of graded representations of $\sltwo$ should agree with the class of $L$ in the skein module of $A\times I$.  The details of this identification and its generalization to other skein modules is something that should be interesting to investigate further.

In fact, the entire story above can be generalized from $\sltwo$ to $\mathfrak{sl}_n$ using other representations of Khovanov-Lauda-Rouquier's categorified quantum groups.  The details of this generalization, including a definition of annular $\mathfrak{sl}_n$ homology and an explicit description of the action of $\mathfrak{sl}_n$ on the annular homology of any link, have been carried out in recent interesting work of Queffelec-Rose \cite{QR}.

\subsection{Organization}
The organization of the paper is as follows.  

\begin{itemize}
\item In Sections \ref{sec:repPrelim} and \ref{sec:Prelim} we recall the definition of sutured Khovanov homology ($\SKh$) for annular links and review some basic facts about the representation theory of $\sltwo$ and $\exsltwo$.

\item In Section \ref{sec:ckhsl2}, we define the $\sltwo$ action on the chain level and prove that the action commutes with the sutured differential, hence induces an action on homology that is diagram-independent. To do this, we reinterpret $\SKh$ of an annular link in terms of Bar-Natan's cobordism category, first extending the Khovanov bracket to the annular setting, then rephrasing the sutured annular Khovanov chain complex $\CKh$ as a functor from the annular Bar-Natan cobordism category to a category of graded representations of $\sltwo$ (Proposition \ref{prop:CKhbracket}).

\item In Section \ref{sec:skhsl2}, we describe some basic properties of sutured annular Khovanov homology as an $\sltwo$ representation. In particular, we prove that it is trapezoidal with respect to the $k$ grading, show its functoriality (up to sign) under annular link cobordisms, and explain how the $\sltwo$ action at the chain level can be understood via the standard action by marked points. 
\item In Section \ref{sec:exsl2}, we enlarge the action of $\sltwo$ on $\SKh(\bL)$ to that of the Lie superalgebra $\exsltwo$, and prove that annular link cobordisms induce well-defined morphisms of $\exsltwo$ modules (Proposition \ref{prop:current}). Theorem \ref{thm:extcurrsl2} follows.

\item In Section \ref{sec:Sn} we prove Theorem \ref{thm:Sn}. 
We also introduce the inductive limits $\SKh^{even}(K)$ and $\SKh^{odd}(K)$, which are infinite-dimensional invariants of the knot $K \subset S^3$.
\item In Section \ref{sec:repexsltwo} we give a quiver description of the category of finite-dimensional representations of $\exsltwo$, showing directly that these categories are governed by finite-dimensional quadratic (in fact Koszul) algebras.

\item In Section \ref{sec:ex} we include some example computations and conjectures.

\item In the appendix, we state and prove the annular version of the Carter-Saito theorem \cite{carter_saito} needed for the functoriality statements in Section \ref{sec:skhsl2}.  
\end{itemize}
\subsection{Acknowledgements}
The authors would like to thank John Baldwin, Jack Hall, Hannah Keese, Aaron Lauda, Scott Morrison, Hoel Queffelec, Peter Samuelson, Alistair Savage, and David Treumann for interesting discussions. The authors would also like to thank Mikhail Khovanov, who originally proposed the idea of using maps induced by Reidemeister moves to define an $\mathfrak{S}_n$--action on the Khovanov homology of $n$--cables of knots.

\section{Representation theoretic preliminaries}\label{sec:repPrelim}
\subsection{$\sltwo$ and its finite-dimensional representations}
We work over $\C$ throughout. Accordingly, we will denote the Lie algebras $\glk(\C)$ and $\slk(\C)$ by $\glk$ and $\slk$, respectively.





We recall some elementary facts about the finite-dimensional representation theory  of the Lie algebra $\sltwo$.
The Lie algebra $\sltwo$ has a $\C$--vector space basis given by the set $\{e,f,h\}$ with Lie bracket:
\begin{equation} \label{eqn:Bracket}
[e,f] = h, \,\, [e,h] = -2e, \,\, [f,h] = 2f.
\end{equation}

With respect to the standard basis of the $2$--dimensional defining representation of $\sltwo$, we have:
\[ h \mapsto \left(\begin{array}{cc} 1 & 0\\
					0 & -1\end{array}\right), \,\, e \mapsto \left(\begin{array}{cc} 0 & 1\\
														0 & 0\end{array}\right), \,\,
				f \mapsto \left(\begin{array}{cc} 0 & 0\\
								1 & 0\end{array}\right).\] 

As a $\C$--vector space, any finite-dimensional representation, $U$, of $\sltwo$ decomposes into {\em weight spaces}, i.e., into eigenspaces for the action of $h$. Explicitly, \[U := \bigoplus_{\lambda \in \Z} U[\lambda],\] where \[
U[\lambda] := \{v \in U\,\,\vline\,\, hv = \lambda v\}.\] The bracket relations tell us that the generators $e$ (resp., $f$) act on the weight spaces as raising (resp., lowering) operators $e_\lambda: U[\lambda] \rightarrow U[\lambda + 2],$ (resp., $f_\lambda: U[\lambda] \rightarrow U[\lambda -2]$).

Each finite-dimensional irrep of $\sltwo$ is determined by its {\em highest weight}, $N \in \Z^{\geq 0}$. Explicitly, for each $N \in \Z^{\geq 0}$, one constructs an $(N+1)$--dimensional irrep, $V_{(N)}$, of the form
\[V_{(N)} := \mbox{Span}_\C\{v, fv, \ldots, f^{N}v\},\] with $hv = Nv$ (and, hence, $h(f^{i}(v)) = (N-2i)f^{i}(v))$. All finite-dimensional $\sltwo$ irreps arise in this manner.  The defining two-dimensional irreducible representation $V_{(1)}$ of $\sltwo$, which plays a central role in what follows, will simply be denoted $V$.

\subsection{The Lie superalgebra $\exsltwo$} \label{sec:exsltwo}
We now describe the exterior current algebra $\exsltwo$ by generators and relations.  We have
$$
	\exsltwo \cong \sltwo \oplus \sltwo,
$$
with the first summand in degree 0 and the second in degree 1 for the $\Z$ (and $\Z_2$) gradings.  We fix the standard $\{e,f,h\}$ basis of $\sltwo$; in order to distinguish the two distinct $\sltwo$ summands in $\exsltwo$ from each other, we will write the standard basis of the degree 1 summand as $\{v_{2},v_{-2}, v_0\}$.  In this basis, the adjoint action of $\sltwo$ action is
$$
	e(v_2) = 0, \ e(v_0) = -2v_2, \ e(v_{-2}) = v_0, 
$$
$$
	f(v_2) = -v_0, \ f(v_0) = 2v_{-2}, \ f(v_{-2}) = 0.
$$
Thus, in the basis $\{e,f,h,v_2, v_{-2}, v_0\}$, the Lie superalgebra $\exsltwo$ has defining relations
\begin{itemize}
\item $[e,f] = h$;
\item $[h,e] = 2e$;
\item $[h,f] = -2f$;
\item $[e,v_{2}] = 0$;
\item $[e,v_0] = -2v_2$;
\item $[e,v_{-2}] = v_0 = -[f,v_2]$;
\item $[f,v_0] = 2v_{-2}$;
\item $[f,v_{-2}] = 0$;
\item $[h,v_2] =2v_2$;
\item $[h,v_0] = 0$;
\item $[h,v_{-2}] = -2v_{-2}$;
\item $[v_i,v_j]=0$ for $i,j\in \{2,0,-2\}$.
\end{itemize}
The $\Z$ and $\Z_2$ gradings on $\exsltwo$ induce $\Z$ and $\Z_2$ gradings on its enveloping algebra $U(\exsltwo)$.


\section{Topological preliminaries}\label{sec:Prelim}

\subsection{Sutured annular Khovanov homology and Lee homology}\label{subs:SKh}
Let $A$ be a closed, oriented annulus, $I = [0,1]$ the closed, oriented unit interval. Via the identification 
\[A \times I= \{(r,\theta,z)\,\,\vline\,\,r \in [1,2], \theta \in [0, 2\pi), z\in [0,1]\} \subset (S^3 = \R^3 \cup \infty) ,\] any link, $\bL \subset A \times I$, may naturally be viewed as a link in the complement of a standardly imbedded unknot, $(U = z$--axis $\cup \,\,\infty) \subset S^3$. Such an {\em annular link} $\bL \subset A \times I$ admits a diagram, $\cP(\bL) \subset A,$ obtained by projecting a generic isotopy class representative of $\bL$ onto $A \times \{1/2\}$, and from this diagram one can construct a triply-graded chain complex, $\CKh(\cP(\bL))$, using a version of Khovanov's original construction \cite{K} due to Asaeda-Przytycki-Sikora \cite{MR2113902} and L. Roberts \cite{GT07060741} (see also \cite{AnnularLinks}), briefly recalled here.

View $\cP(\bL) \subset A$ instead as a diagram on $S^2 - \{\XX,\OO\}$, where $\XX$ (resp., $\OO$) are basepoints on $S^2$ corresponding to the inner (resp., outer) boundary circles of $A$. If we temporarily forget the data of $\XX$, we may view $\cP(\bL)$ as a diagram on $\R^2 = S^2 - \{\OO\}$ and form the ordinary bigraded Khovanov complex
\[\CKh(\cP(\bL)) = \bigoplus_{(i,j) \in \Z^2} \CKh^i(\cP(\bL);j)\] as described in \cite{K}.

Recalling that the generators of $\CKh(\cP(\bL))$ correspond to {\em oriented} Kauffman states (cf. \cite[Sec. 4.2]{JacoFest}), we now obtain a third grading on the complex by defining the ``$k$" grading of a generator (up to an overall shift) to be the algebraic intersection number of the corresponding oriented Kauffman state with a fixed oriented arc $\gamma$ from $\XX$ to $\OO$ that misses all crossings of $\cP(\bL)$. Roberts proves (\cite[Lem. 1]{GT07060741}) that the Khovanov differential, $\partial$, is non-increasing in this extra grading. Decomposing $\partial = \partial_0 + \partial_-$ into its $k$-grading--preserving and $k$-grading--decreasing parts, we obtain a triply-graded chain complex $(\CKh(\cP(\bL)), \partial_0)$ whose homology, 

\[\SKh(\bL) := \bigoplus_{(i,j,k) \in \Z^3} \SKh^{i}(\bL;j,k),\] is an invariant of $\bL \subset A \times I$, called the {\em sutured annular Khovanov homology} of $\bL$.  More can be said:

\begin{lemma}
Let $\CKh(\cP(\bL))$ be the triply-graded vector space associated to a diagram of an annular link, $\bL \subset A \times I$ as above, and let $\partial = \partial_0 + \partial_-$ be the decomposition of the Khovanov differential in terms of the $k$--grading. Then $(\CKh(\cP(\bL)), \partial_0,\partial_-)$ is a bicomplex.
\end{lemma}

\begin{proof}
The operator $\partial_-$ is homogeneous (of degree $-2$) in the $k$--grading. Decomposing $\partial^2 = 0$ into its homogeneous summands, it follows that $\partial_-^2 = 0$ and $\partial_0\partial_- + \partial_-\partial_0 = 0$.
\end{proof}

One therefore obtains a spectral sequence converging to $\Kh(\bL)$ whose $E^1$ page is $\SKh(\bL)$. Each page of this spectral sequence is an invariant of $\bL \subset A \times I$ (cf. \cite{GT07060741}).

The reader is warned that the other spectral sequence associated to this bicomplex (whose $E^1$ page is the homology of $(\CKh(\cP(\bL), \partial_-)$) is {\em not} an invariant of the annular link $\bL$.

\begin{remark} \label{rmk:filtration} In \cite{GT07060741}, the complex $\CKh(\cP(\bL))$ is considered as a \emph{filtered} complex, with the filtration induced by the $k$ grading.  This filtration agrees with the standard one associated to the bicomplex described above.
\end{remark}

\begin{remark} \label{rmk:quantumgr} In what follows it will be convenient for us to replace Khovanov's original ``$j$" (quantum) grading with a ``$\,j'\,$" (filtration-adjusted quantum) grading. If ${\bf x} \in \CKh(\cP(\bL))$ is a generator, $j'({\bf x}) := j({\bf x}) - k({\bf x}).$ The sutured differential, $\partial_0$, is degree $(1,0,0)$ with respect to the $(i,j',k)$ grading, while the endomorphism $\partial_-$ has degree $(1,2,-2)$.
\end{remark}

The chain complex $\CKh(\cP(\bL))$ also comes equipped with a natural involution $\Theta$, defined in the following lemma (cf. \cite[Prop.7.2, (3)]{QuiverAlgebras}):

\begin{lemma} \label{lem:inv} Let $\bL \subset (A \times I) \subset S^3$ be an annular link, \[\cP(\bL) \subset (S^2 - \OO - \XX) \subset (S^2 - \OO) \sim \R^2\] a diagram for $L$, and \[\cP'(\bL) \subset (S^2 - \XX - \OO) \subset (S^2 - \XX) \sim \R^2\] the diagram obtained by exchanging the roles of $\OO$ and $\XX$. The corresponding map \[\Theta: \CKh(\cP(\bL)) \rightarrow \CKh(\cP'(\bL))\] is a chain isomorphism inducing an isomorphism $\SKh^{i,j'}(L;k) \cong \SKh^{i,j'}(L;-k)$ for all $(i,j',k) \in \Z^3$.
\end{lemma}

\begin{proof}
Recall that the generators of the sutured Khovanov complex are identified with enhanced (oriented) Kauffman states. Therefore, the result of preserving the orientation on $S^2$ but exchanging the roles of $\OO$ and $\XX$ is that a counterclockwise (resp., clockwise) orientation on a {\em nontrivial} circle is now viewed as a clockwise (resp., counterclockwise) orientation. On the other hand, orientations on all {\em trivial} components are preserved. In the language of \cite{GT07060741}, $v_+$ and $v_-$ labels are exchanged, but $w_{\pm}$ labels are preserved.
Since the sutured Khovanov differential is symmetric with respect to $v_{\pm}$ (cf. \cite[Sec. 2]{GT07060741}), $\Theta$ is a chain map. Moreover, $\Theta \circ \Theta = \Id$, so it is a chain isomorphism. That it preserves the homological ($i$) and new quantum ($j'$) gradings but changes the sign of the weight space ($k$) grading is immediate from the definition.
\end{proof}

Let $\partial^{Lee}$ denote Lee's deformation of Khovanov's differential, defined in \cite[Sec. 4]{Lee} (and denoted $\Phi$ there).  Lee proves:
\begin{itemize}
	\item $(\partial^{Lee})^2 = 0$\\
	\item $\partial \partial^{Lee} + \partial^{Lee}\partial = 0$.\\
\end{itemize} 

As with the Khovanov differential above, we may write the Lee deformation as a sum
$$
	\partial^{Lee} = \partial^{Lee}_0 + \partial^{Lee}_+,
$$
where this time $\partial^{Lee}$ has $(i,j',k)$ degree $(1,4,0)$, while $\partial^{Lee}_+$ has degree $(1,2,2)$.

We collect the relationships between $\partial_0,\partial_-,\partial^{Lee}_0,\partial^{Lee}_+$ and $\Theta$ in the following:

\begin{lemma}\label{lem:partials}
The endomorphisms $\partial_0,\partial_-,\partial^{Lee}_0,\partial^{Lee}_+$ and $\Theta$ of $\CKh(\cP(\bL))$ satisfy the following.
\begin{enumerate}
\item $\partial_0^2 = (\partial^{Lee}_0)^2= 0$;
\item $\partial_-^2 = (\partial^{Lee}_+)^2 = 0$;
\item $\Theta\partial_0 = \partial_0 \Theta$;
\item $\Theta\partial^{Lee}_0 = \partial^{Lee}_0 \Theta$;
\item $\Theta \partial_- = \partial^{Lee}_+ \Theta$;

\item $\partial_- \partial_0 + \partial_0\partial_- = 0$;
\item $\partial^{Lee}_+ \partial^{Lee}_0 + \partial_0\partial^{Lee}_+ = 0$;
\item $\partial^{Lee}_+ \partial_0 + \partial_0\partial^{Lee}_+ = 0$;
\item $\partial_- \partial^{Lee}_0 + \partial^{Lee}_0\partial_- = 0$

\item $\partial_0\partial^{Lee}_0 + \partial^{Lee}_0 \partial + \partial_- \partial^{Lee}_+ + \partial^{Lee}_+ \partial_- = 0$;
\end{enumerate}
\end{lemma} 
\begin{proof}
The first observations are that $\Theta$ commutes with $\partial_0$ and $\partial^{Lee}_0$, and that conjugation by $\Theta$ exchanges $\partial_-$ and $\partial^{Lee}_+$; these statements are direct checks along the various split and merge maps in the cube.
The remaining statements now follow from this by expanding the equations $\partial^2 = 0$, $(\partial^{Lee})^2=0$, and $\partial \partial^{Lee} + \partial^{Lee}\partial = 0$ into $k$--homogeneous terms.
\end{proof}

\section{A sutured annular Khovanov bracket and $\sltwo$}\label{sec:ckhsl2}

\subsection{Khovanov bracket for annular links}\label{subs:bracket}
Let $\bL\subset A\times I$ be an annular link and $\cP(\bL)\subset A$ a diagram for $\bL$, as in the previous subsection. Following Bar-Natan \cite[Sections 2 and 11]{MR2174270}, one can define an abstract chain complex
\[
[\cP(\bL)]=\left(\ldots\longrightarrow[\cP(\bL)]^{i-1}\longrightarrow [\cP(\bL)]^{i}\longrightarrow[\cP(\bL)]^{i-1}\longrightarrow\ldots\right)
\]
by constructing a resolution cube for $\cP(\bL)$ and then formally taking direct sums of resolutions that sit in the same ``$i$'' degree. The differential in this complex is defined in terms of (signed) saddle cobordisms associated to the edges of the resolution cube, and the resulting complex, $[\cP(\bL)]$, is viewed as an
object in the category $\komh(\mat(\cobal))$, defined below.

\begin{definition}
Let $\coba$ denote the category whose objects are closed, unoriented $1$-manifolds in $A$, and whose morphisms between two objects $C_0$ and $C_1$ are unoriented $2$-cobordisms $S\subset A\times I$ satisfying $\partial S=(C_0\times\{0\})\amalg(C_1\times\{1\})$, considered up to isotopy rel boundary. Let $\cobal$ denote the category which has the same objects as $\coba$ and whose morphisms are formal $\C$--linear combinations of morphisms in $\coba$, considered modulo the $S$, $T$, and $4Tu$ relations described in \cite[Subs. 4.1]{MR2174270}.
\end{definition}

\begin{definition} For a pre-additive category $\mathcal{A}$, we denote by $\mat(\mathcal{A})$ the additive closure of $\mathcal{A}$. If $\mathcal{A}$ is an additive category, then we denote by $\komh(\mathcal{A})$ the bounded homotopy category of $\mathcal{A}$.
\end{definition}

We will use the shorthand notation $\kobha:=\komh(\mat(\cobal))$, and we will write $\kobpha$ for the (non-additive) category obtained from $\kobha$ by identifying each morphism with its negative. Bar-Natan proved in \cite{MR2174270} that the homotopy type of the complex $[\cP(\bL)]$ is invariant under Reidemeister moves, and thus the object $[\cP(\bL)]\in\kobha$ provides an invariant for the annular link $\bL\subset A\times I$ when considered up to isomorphism in $\kobha$. We will see in Proposition~\ref{prop:functorial} below that $[\cP(\bL)]$ also has good functoriality properties.

\begin{remark}
$\cobal$ can be transformed into a graded category by replacing objects of $\cobal$ by pairs $(C,j')$ where $C\in\cobal$, and where $j'$ is an integer, to be thought of as a formal grading shift. In the remainder of this paper, we will implicitly assume that $\cobal$ is this graded category, and that $[\cP(\bL)]$ is the graded version of the annular Khovanov bracket (defined as in \cite[Sec. 6]{MR2174270}). 
\end{remark}


\subsection{$\CKh$ as a TQFT valued in representations of $\sltwo$}\label{subs:annulartqft}

Let $\rep$ denote the category of $\Z$-graded representations of $\sltwo$.  An object of $\rep$ is a direct sum
$$Y = \bigoplus_{n\in \Z} Y(n),$$
where each $Y(n)$ is a finite-dimensional representation of $\sltwo$.  
An object $Y\in \rep$ is sometimes naturally regarded as a \emph{bigraded} vector space, where the component gradings are the $\Z$-grading above and the $\sltwo$ weight space grading.  These component gradings will be referred to as the $j'$ and $k$ gradings, respectively (in particular, $k$-grading means $\sltwo$-weight-space grading). For $m\in\Z$, we will denote by $\{m\}$ the grading shift operator which acts on objects of $\rep$ by raising the $j'$ grading by $m$. That is, if $Y$ is an object of $\rep$, then $Y\{m\}$ denotes the object with components $(Y\{m\})(n+m):=Y(n)$.

We will define a (1+1) dimensional TQFT with values in $\rep$.  In order to do this, we will need to use three particular graded representations of $\sltwo$:
\begin{itemize}
\item Let \[V := \mbox{Span}_{\C}\{v_+,v_-\}\]
denote the two dimensional defining representation of $\sltwo$.  The bigrading on $V$ is $j'(v_\pm)=0$ and $k(v_\pm)=\pm 1$; in particular, $v_+$ is a highest weight vector, and $v_- = f \cdot v_+$ a lowest weight vector.  
\item Let \[\oV := \mbox{Span}_{\C}\{\ov_+,\ov_-\}\]
denote the dual representation to $V$, where $\ov_-$ is the dual vector to $v_+$ and $\ov_+$ is the dual vector to $v_-$.  The bigrading on $\oV$ is 
$j'(\ov_\pm)=0$ and $k(\ov_\pm)=\pm 1$.
\item Let \[W := \mbox{Span}_{\C}\{w_+,w_-\}\]
be the trivial two-dimensional representation of $\sltwo$, graded with $j'(w_\pm)=\pm 1$ and $k(w_\pm)=0$. 
\end{itemize}
Of course, the objects $V$ and $\oV$ are isomorphic in $\rep$.  However, the matrices for the action of the standard basis $\{e,f,h\}$ with respect to the bases
$\{v_{\pm}\}$ on $V$ and $\{\ov_{\pm}\}$ on $\oV$ are different; for example,
$$
	e\cdot v_- = v_+, \text{ but } e\cdot \ov_- = -\ov_+.
$$
(Note also that Khovanov's ``$j$'' grading used in \cite{GT07060741} is the sum of the ``$j'$'' and the ``$k$'' grading. See Remark \ref{rmk:quantumgr})

We will now define an additive functor
\[
\cF\colon\cobal\longrightarrow\rep.
\]
\subsubsection{$\cF$ on objects} \label{sec:FunctorOb}
Let $C\in\cobal$ be an unoriented $1$-manifold $C\subset A$ with $\ell_n$ nontrivial circles and $\ell_t$ trivial circles. Choose any ordering $C_1, \ldots, C_{\ell_n}, C_{\ell_n +1}, \ldots, C_{\ell_n + \ell_t}$ of the circles of $C$ such that all of the nontrivial circles are listed first.
Regard $C$ as a submanifold of $S^2-\XX-\OO$, where $\XX$ (resp., $\OO$) is a basepoint on $S^2$ corresponding to the inner (resp., outer) boundary of $A$, as in Subsection~\ref{subs:SKh}. For $i\in\{1,\ldots,\ell_n\}$,  we denote by $X(C_i)\in\{0,\ldots,\ell_n-1\}$ the number of nontrivial circles of $C$ which lie in the same component of $S^2-C_i$ as the basepoint $\XX$, and we define
\[ \epsilon(C_i):=(-1)^{X(C_i)}.\]

We now set
\[
\mathcal{F}(C):= (\bigotimes_{\epsilon(C_i) = 1} V) \otimes (\bigotimes_{\epsilon(C_i) = -1} \oV)\otimes (\bigotimes_{s = 1}^{\ell_t} W).
\]
Thus nontrivial circles $C_i$ are assigned either $V$ or $\oV$, depending on the sign $\epsilon(C_i)$, and trivial circles are assigned $W$.

\subsubsection{$\cF$ on morphisms} \label{sec:FunctorMorph}

To define $\mathcal{F}$ on morphisms, we use that morphisms of $\cobal$ are generated by elementary Morse cobordisms: cup cobordisms creating a trivial circle, cap cobordisms annihilating a trivial circle, and saddle cobordisms, which either merge two circles into one or split one circle into two. To saddle cobordisms, we now assign the merge/split maps defined by L.~Roberts in \cite[Sec. 2]{GT07060741}; to cup cobordisms, we assign the map $\iota\colon\C\rightarrow W$ given by $\iota(1):=w_+$; to cap cobordisms, we assign the map $\epsilon\colon W\rightarrow\C$ given by $\epsilon(w_+):=0$ and $\epsilon(w_-):=1$.

There are two points about the above definition that require explanation.  The first is that the above assignment to cups, caps, and saddles induces a well-defined linear map on any annular cobordism.  To see this, note that $\iota$ and $\epsilon$ are precisely the unit and counit maps defined by Khovanov in \cite{K}. L.~Roberts further shows that his merge/split maps can be viewed as the degree 0 parts (with respect to the ``$k$'' filtration) of Khovanov's multiplication/comultiplication maps. Hence it follows that $\mathcal{F}$ can be viewed as the homogeneous part (with respect to the ``$k$'' grading) of Khovanov's $(1+1)$-dimensional TQFT from \cite{K}. In particular, this shows that the linear map $\mathcal{F}$ assigns to a cobordism does not depend on how it is assembled from elementary Morse cobordisms.  

The second point is to observe that the linear maps that $\mathcal{F}$ assigns to generating morphisms of $\cobal$ (cup, cap, and saddle cobordisms in $A\times I$) are maps of $\sltwo$-modules, so that the functor $\cF$ can be regarded as taking values in $\rep$. For cup and cap cobordisms in $A\times I$, this is obvious, because the ``non-identity parts'' of these cobordisms only involve trivial circles, and hence the ``non-identity parts'' of the associated linear maps only involve $W$ factors, on which the $\sltwo$ action is trivial. For saddle cobordisms in $A\times I$, we have the following lemma:

\begin{lemma} \label{lem:saddle} With the above assignment of $\sltwo$--module structures to the vector spaces $\mathcal{F}(C)$, each merge/split map (defined as in \cite[Sec. 2]{GT07060741}):
\begin{eqnarray*}
	W \otimes W &\longleftrightarrow& W,\\
	W \otimes V &\longleftrightarrow& V, \\
	W \otimes \oV &\longleftrightarrow& \oV, \\
	V \otimes \oV &\longleftrightarrow& W
\end{eqnarray*}
is an $\sltwo$--module map of $(j',k)$--bidegree $(-1,0)$.
\end{lemma}

\begin{proof} 
Since $W$ is a direct sum of trivial $\sltwo$--modules, there is nothing to check for line (1). 
Lines (2) and (3) have essentially the same proof. For example, in line (2) we have 
$$
	{\oV}^{\otimes x}\otimes V^{\otimes y} \otimes (V \otimes W) \otimes W^{\otimes z} \xrightarrow{\Id \otimes \ldots \Id \otimes (\Phi) \otimes \Id}
	{\oV}^{\otimes x}\otimes V^{\otimes y} \otimes (V) \otimes W^{\otimes z},
$$ 
where $\Phi$ is either the merge or split map, depending on the direction of the arrow, and \[V \otimes W := V \{-1\} \oplus V\{1\}\]
is a direct sum of two irreducible graded representations of $\sltwo$.


If $C$ (resp., $C'$) is the nontrivial circle involved in the merge/split before (resp., after) the merge/split, then it is straightforward to check that $\epsilon(C) = \epsilon(C')$. It follows that the Roberts merge (resp., split) map is precisely the canonical degree $(0,0)$ projection of $\sltwo$ representations:
\[V\{-1\} \oplus V\{1\} \longrightarrow V\{1\}\] (resp., inclusion):
\[V \longrightarrow V\{0\} \oplus V\{2\}.\]

For line (4), we have 
$$
	{\oV}^{\otimes x}\otimes V^{\otimes y} \otimes (V \otimes \oV) \otimes W^{\otimes z} \xrightarrow{\Id \otimes \ldots \Id \otimes (\Phi) \otimes \Id}
	{\oV}^{\otimes x}\otimes V^{\otimes y} \otimes (W) \otimes W^{\otimes z},
$$
where $\Phi$ again denotes the merge or split, depending on the direction of the arrow, and
\[V \otimes \oV := V_{(0)} \oplus V_{(2)}\]
is the decomposition into the irreducible trivial ($V_{(0)}$) and adjoint ($V_{(2)}$) $\sltwo$ representations.  Let $C_1$ and $C_2$ denote the two nontrivial circles involved in the merge. Then $-\epsilon(C_1) = \epsilon(C_2)$. 
Moreover, with respect to the chosen bases of $V, V^*$ we have:
\[V_{(0)} = \mbox{Span}_\C\{v_+ \otimes \ov_- + v_- \otimes \ov_+\} \subset V\otimes \oV.\]  We conclude that the merge and split maps are nonzero scalar multiple of the composition of inclusion and projection maps
\[
	V_{(0)} \longleftrightarrow V\otimes \oV.
\]


We have shown that each vector space $\mathcal{F}(C)$ can be given the structure of a graded $\sltwo$--module, and that the maps associated to morphisms of $\cobal$ intertwine the $\sltwo$ actions. Hence $\mathcal{F}$ lifts to a functor with values in $\rep$, as desired.
\end{proof}

\begin{remark}\label{rem:gltwo}
In fact the functor $\mathcal{F}$ can take values in the category of graded representations of $\gltwo$, after declaring $V$ to be the defining 
two-dimensional representation of $\gltwo$, $\oV$ the linear dual, and $W$ the trivial two-dimensional representation.  In some sense, the distinction between $V$ and $\oV$ in the construction is more natural when $\mathcal{F}$ takes values in $\mbox{gRep}(\gltwo)$, since $V$ and $\oV$ are no longer isomorphic as $\gltwo$ representations.  On the other hand, the exterior current algebra which appears later in the paper is that of $\sltwo$, not $\gltwo$.   
\end{remark}


We now have the following.

\begin{proposition}\label{prop:CKhbracket} The sutured annular Khovanov complex can be obtained from the annular Khovanov bracket by applying the functor $\mathcal{F}$:
\[(\CKh(\cP(\bL)),\partial_0)\cong\mathcal{F}([\cP(\bL)]).\]
\end{proposition}

\begin{proof} This follows immediately from the definition and properties of $\mathcal{F}$ and from the definitions of $[\cP(\bL)]$ of $(\CKh(\cP(\bL)),\partial_0)$ given in  \cite{MR2174270} and \cite{GT07060741}, respectively. 
\end{proof}


The above proposition implies that the sutured annular Khovanov complex of $\cP(\bL)$ can be viewed as a complex in the category $\rep$. We can refine this result by introducing the Schur algebra
\[
S(2,n):=\operatorname{im}(\rho_n),
\]
where $\rho_n\colon U(\sltwo)\rightarrow \mbox{End}_{\mathbb{C}}(V_{(1)}^{\otimes n})$ denotes the usual representation of $U(\sltwo)$ on the $n$th tensor power of the defining representation of $\sltwo$.


\begin{proposition}\label{prop:Schur}
If there exists an essential arc $\gamma\subset A$ intersecting $\cP(\bL)\subset A$ transversely in exactly $n$ points, none of which are crossings of $\cP(\bL)$, then the $\sltwo$ action on $\CKh(\cP(\bL))$ factors through the Schur algebra $S(2,n)$.
\end{proposition}

\begin{proof} Suppose there is an arc $\gamma$ as in the proposition. Then the number $\ell_n$ of nontrivial circles in any given resolution $C$ of $\cP(\bL)$ satisifies \[0\leq\ell_n\leq n\quad\mbox{and}\quad\ell_n\equiv n\pmod 2, \] and this implies that the representation $V_{(1)}^{\otimes\ell_n}$ appears in $V_{(1)}^{\otimes n}$ as a subrepresentation. It follows that elements of $\operatorname{ker}(\rho_n)\subset U(\sltwo)$ act trivially on $V_{(1)}^{\otimes\ell_n}$, and since $\mathcal{F}(C)\cong V^{\otimes \ell_n}\otimes W^{\otimes\ell_t}$ is isomorphic to a direct sum of $2^{\ell_t}$ copies of $V_{(1)}^{\otimes \ell_n}$, this means that the $\sltwo$ action on $\mathcal{F}(C)$ factors through $U(\sltwo)/\operatorname{ker}(\rho_n)\cong\operatorname{im}(\rho_n)=S(2,n)$.
\end{proof}

Passing to homology, Propositions \ref{prop:CKhbracket} and \ref{prop:Schur} immediately imply the following.
\begin{proposition}\label{prop:homology}
We have $\SKh^i(\cP(\bL))\cong H_i(\cF([\cP(\bL)])) \in \rep$ for all $i$, so that $\SKh(\cP(\bL))$ is a bigraded representation of $\sltwo$.  The isomorphism type of this bigraded representation is an invariant of the isotopy class of the annular link $\bL$.  The action of $\sltwo$ on $\SKh(\bL)$ factors through the Schur algebra $S(2,n)$ where $n$ is the wrapping number of $\bL$.
\end{proposition}

Here, the wrapping number of an annular link $\bL\subset A\times I$ is defined as the smallest integer $n\geq 0$ such that there exists an arc $\gamma\subset A$ as in Proposition~\ref{prop:Schur}, where the minimum is taken over all diagrams $\cP(\bL)\subset A$ representing the annular link $\bL$.

The above proposition will be strengthened by enlarging the action of $\sltwo$ to an action of $\exsltwo$ in Proposition \ref{prop:current}, which implies Theorem \ref{thm:extcurrsl2}.

\section{Basic properties of $\SKh$ as an $\sltwo$ representation}\label{sec:skhsl2}

The fact that the sutured Khovanov homology of an annular link is an $\sltwo$ representation implies that is {\em trapezoidal} with respect to the $\sltwo$ weight space grading, which we have seen agrees with the $k$ (filtration) grading. We therefore have the following immediate consequence of Proposition \ref{prop:homology}.

\begin{corollary} \label{cor:trap}
Let $\bL \subset A \times I$ be an annular link. Then \[\mbox{dim}_\C(\SKh(\bL;k)) \geq \mbox{dim}_\C(\SKh(\bL;k'))\] whenever $k \equiv k' \mod 2$ and $|k| \leq |k'|$.
\end{corollary}

Moreover, the $\sltwo$ representation structure on the sutured Khovanov homology of an annular link gives an alternative way to understand its symmetry  (Lemma \ref{lem:inv}) with respect to the $k$--grading. Indeed, the following lemma is readily seen from the chain level definitions of the raising operator $e$, the lowering operator $f$, and the involution $\Theta$ on the sutured chain complex associated to an annular link.

\begin{lemma} \label{lem:eThetaf} Let $\bL \subset A \times I$ be an annular link and let $e,f, \Theta$ be the endomorphisms on $\SKh(\bL)$ described above. Then $e = \Theta f \Theta.$
\end{lemma}

\subsection{Functoriality for Annular Link cobordisms}
The $\sltwo$--representation structure is functorial with respect to annular link cobordisms. To state this precisely, we must first introduce the following closely-related topological categories.
\begin{definition} Let $\cob$ denote the category of {\em annular link cobordisms}. The objects of $\cob$ are oriented annular links in general position (i.e., the projection to $A \times \{1/2\}$ is a diagram). A morphism between links $\bL_0$ and $\bL_1$ is a smoothly imbedded oriented surface, $F \subset (A \times I) \times I$ satisfying $\partial F = -(\bL_0\times\{0\}) \amalg (\bL_1\times\{1\})$, considered modulo isotopy rel boundary.
\end{definition}

\begin{definition}
Let $\comb$ denote the category of {\em combinatorial} annular link cobordisms.  The objects of $\comb$ are annular link diagrams, considered up to planar isotopy. The morphisms are Carter-Saito movies \cite{carter_saito}, specified by finite sequences of link diagrams, each related to the next by a Reidemeister move, an elementary Morse move, or a planar isotopy in $A$.
\end{definition}

We have a canonical functor $\mathcal{L}\colon\comb \rightarrow \cob$ obtained by lifting each annular link diagram to a specific annular link in general position and each Carter-Saito movie to a specific annular link cobordism with Morse decomposition described by the movie. The following is an annular version of  \cite[Thm. 4]{MR2174270}:

\begin{proposition} \label{prop:functorial}
The assignment $\cP(\bL) \mapsto [\cP(\bL)]$ extends to a functor \[[-]: \comb \longrightarrow \kobha.\]
Up to signs, this functor factors through the category $\cob$ via the canonical functor $\mathcal{L}\colon\comb\rightarrow\cob$. In particular, $[-]$ descends to a functor $\cob\rightarrow\kobpha$.
\end{proposition}

\begin{proof} On generating morphisms of $\comb$, we define the functor $[-]$ as follows:
\begin{itemize}
\item To movies representing Reidemeister moves, we assign the homotopy equivalences constructed in \cite{MR2174270} within the proof of the invariance theorem (\cite[Thm. 1]{MR2174270}).
\item To movies representing elementary Morse cobordisms, we assign the chain maps obtained by interpreting these Morse cobordisms  as the corresponding generating morphisms of $\cobal$.
\item
For movies representing planar isotopies in $A$, we use an analogous definition.
\end{itemize}

In the appendix, we adapt the Carter-Saito theorem \cite{carter_saito} to the annular setting. That is, we show that every (smooth) annular link cobordism can be presented by an annular Carter-Saito movie, and that two annular Carter-Saito movies represent isotopic annular link cobordisms if and only if they can be transformed one to the other by a finite sequence of annular Carter-Saito movie moves. Bar-Natan already proved \cite{MR2174270} that the chain maps associated to Carter-Saito movies are invariant under movie moves when considered up to sign and homotopy, and so it follows that the functor $[-]$ descends to a functor $\cob\rightarrow\kobpha$, as desired.
\end{proof}

\subsection{The $\sltwo$ action via marked points} \label{sec:markpt}

Let $\mathcal{P}(\bL) \subset S^2 - \OO - \XX$ be a diagram of a link $\bL \subset A \times I \subset S^3$ and suppose $p_1, \ldots, p_n \subset \mathcal{P}(\bL)$ is a collection of $n$ distinct marked points on $\mathcal{P}(\bL)$ in the complement of a neighborhood of the crossings. Temporarily forgetting the data of the basepoint $\XX$, recall \cite{K, MR2034399, HeddenNi} that we have an action of \[\cA_n := \C[x_1, \ldots, x_n]/(x_1^2, \ldots, x_n^2)\] on the chain complex $CKh(\mathcal{P}(\bL))$ defined as follows. 

Let $\mathcal{P}'(\bL)$ denote the diagram obtained from $\mathcal{P}(\bL)$ by placing, for each $i \in \{1, \ldots, n\}$ a tiny trivial circle $C_i$ in a region adjacent to $p_i$. We then have \[\CKh(\mathcal{P}'(\bL)) \cong \CKh(\mathcal{P}(\bL)) \otimes \cA_n,\] along with a map \[m: \CKh(\mathcal{P}(\bL)) \otimes \cA_n \rightarrow \CKh(\mathcal{P}(\bL))\] realized as the composition of the $n$ (commuting) multiplication maps associated to merging $\cP(\bL)$ with $C_1, \ldots, C_n$ at $p_1, \ldots, p_n$.


\begin{proposition} Suppose $\bL \subset A \times I \subset S^3$ is an annular link with diagram $\mathcal{P}(\bL) \subset S^2 - \OO - \XX \subset S^2 - \OO$. Let $\gamma$ be any arc from $\XX$ to $\OO$ that misses all crossings of $\mathcal{P}(\bL)$ and intersects $\mathcal{P}(\bL)$ transversely in $n$ points $p_1, \ldots, p_n$ (whose ordering is determined by the orientation of $\gamma$). Consider the action of $\cA_n := \C[x_1, \ldots, x_n]/(x_1^2, \ldots, x_n^2)$ on $\CKh(\mathcal{P}(\bL))$ induced by the marked points $p_1, \ldots, p_n$.

\begin{enumerate}
	\item Let $f: \CKh(\cP(\bL)) \rightarrow \CKh(\cP(\bL))$ be the lowering operator of the $\sltwo$ action described in Section \ref{subs:annulartqft}. Then \[f = \pm \sum_{i=1}^n (-1)^i x_i.\]
	\item Let $e: \CKh(\cP(\bL)) \rightarrow \CKh(\cP(\bL))$ be the raising operator of the $\sltwo$ action described in Section \ref{subs:annulartqft}. Then \[e = \pm\Theta\left(\sum_{i=1}^n (-1)^{i}x_i\right)\Theta,\] where $\Theta$ is the involution described in Lemma \ref{lem:inv}.
\end{enumerate}
\end{proposition}

\begin{proof} We verify statement (1) by showing that the chain-level map $\sum_{i=1}^n (-1)^i x_i$ corresponding to any arc $\gamma$ from $\XX$ to $\OO$ agrees with the chain-level map $f$ described in Section \ref{subs:annulartqft}.

To see this, let $\cP_\cI(\bL)$ be a resolution of $\mathcal{P}(\bL)$. Recall from Section \ref{subs:annulartqft} that the action of the lowering operator $f$ on the vector space associated to $\cP_\cI(\bL)$ is the standard tensor product representation of the actions of $f$ on the vector spaces associated to each circle of the resolution considered separately. 

Now suppose $C$ is a nontrivial circle of $\cP_\cI(\bL)$. Any arc $\gamma$ as above will then intersect $C$ in an odd number of points: $p_{i_1}, \ldots, p_{i_k}$ according to their order of intersection with the oriented arc $\gamma$. As the actions of $x_{i_1}, \ldots, x_{i_k}$ on the vector space associated to $\cP_\cI(\bL)$ all agree, we then have
\[\sum_{j=1}^k (-1)^{i_j} x_{i_j} = (-1)^{i_1} x_{i_1}.\] Since $i_1$ agrees, mod $2$, with the number of nontrivial circles separating $C$ from $\XX$, the above agrees with the action of $f$ on the vector space associated to $C$ described in Section \ref{subs:annulartqft}. 

Similarly, if $C$ is a trivial circle of $\cP_\cI(\bL)$, any arc $\gamma$ will intersect $C$ in even number of points $p_{i_1}, \ldots, p_{i_k}$, so the action of $\sum_{j=1}^k (-1)^{i_j} x_{i_j}$ on $C$ is $0$, agreeing with the action from Section \ref{subs:annulartqft}. This concludes the proof of statement (1), and statement (2) now follows from Lemma \ref{lem:eThetaf}.
\end{proof}

\begin{remark} The reader is warned that although the chain-level maps $x_i$ commute with the {\em ordinary} Khovanov differential $\partial$, they do not commute with the {\em sutured} Khovanov differential $\partial_0$. On the other hand, their alternating sum commutes with $\partial_0$, as does the involution $\Theta$, but $\Theta$ does not commute with $\partial$.
\end{remark}

\begin{remark} Recall that it is shown in \cite[Prop. 2.2]{HeddenNi} that the chain-level action of $\cA_n$ described above induces a well-defined action (modulo signs) of \[\mathcal{A}_\ell := \C[X_1, \ldots, X_\ell]/(X_1^2, \ldots, X_\ell^2)\] on the ordinary Khovanov homology of an $\ell$--component link (one need only ensure that each link component contains at least one marked point). In particular, $X_i$ denotes the map induced on the ordinary Khovanov homology of $\bL$ by (any one of) the basepoint(s) marking the $i$th component of $\bL$.

Recalling that there is a spectral sequence relating $\SKh(\bL \subset A \times I)$ to $\Kh(\bL \subset S^3)$, it is tempting to conclude that the map induced by the lowering operator $f$ on $\Kh(\bL)$  agrees with $g = \sum_{i=1}^n \epsilon_i X_i$ for some choices $\epsilon_i \in \{\pm 1\}$.

However, the $E^\infty$ page of the spectral sequence is the {\em associated graded} of $\Kh(\bL)$ with respect to the induced filtration. As a result, we can only conclude the weaker statement that the highest-degree terms of the maps agree. More precisely, noting that $f$ is a filtered map of degree $-2$ on the filtered complex described in Remark \ref{rmk:filtration}, we can regard $f$ as a filtration-{\em preserving} map  \[f: \CKh(\bL) \rightarrow \CKh(\bL)\{\{2\}\},\] where in the above $\{\{ n \}\}$ is the operator that shifts $k$--gradings (and hence the induced filtration) up by $2$. Let $f_\infty$ denote the map induced by $f$ on the $E^\infty$ page of the spectral sequence associated to the $k$--filtration, and decompose $g = g_{-2} + g_{-4} + \ldots$ into its $k$--homogeneous terms with respect to the induced $k$--grading on the $E^\infty$ page. Then \[f_\infty = g_{-2}.\]
\end{remark}

\section{$\SKh$ and the current algebra $\exsltwo$}\label{sec:exsl2}

In this section we extend the action of $\sltwo$ on $\SKh(\bL)$ to an action of the exterior current algebra, $\exsltwo$. Note that--in contrast to the $\sltwo$ action--the $\exsltwo$ relations hold at the chain level {\em only up to homotopy.} In what follows, we will construct an action of a slightly larger Lie superalgebra, $\exsltwo_{dg}$, on the sutured annular chain complex, then show  that it induces an action of $\exsltwo$ on the homology. To make these statements precise, we first review some algebra.
\subsection{Chain complexes and Lie superalgebras}
Let $(C^\bullet,\partial)$ be a $\Z$-graded chain complex.  The $\Z$-grading $C^\bullet = \bigoplus_{i\in \Z} C^i$ induces a $\Z_2$ grading
$$
	C = C^{even}\oplus C^{odd}, \text{ where }
$$
$$
	C^{even} = \bigoplus_{n\in \Z} C^{2n} \text{ and } C^{odd} = \bigoplus_{n\in \Z} C^{2n+1},
$$ and hence the structure of a super vector space. Indeed, $C^\bullet$ canonically has the structure of a Lie superalgebra representation.


Let $\End(C^\bullet)$ denote the hom complex of $C^\bullet$, which is a $\Z$-graded super vector space in its own right:
$$
	\End(C^\bullet) = \bigoplus_{n\in \Z} \End^n(C^\bullet),
$$
$$
	\End^n(C^\bullet) =
	\{f: C^\bullet \rightarrow C^{\bullet + n}\}
$$
The differential $\partial \in \End(C)$ is a degree one endomorphism.  We endow $\End(C^\bullet)$ with the structure of a Lie superalgebra by declaring, for $f\in \End^n(C^\bullet)$ and $g\in \End^m(C^\bullet)$,
$$
	[f,g] = f g - (-1)^{nm} g f.
$$
The superalgebra $\End(C^\bullet)$ is also a chain complex, with differential
$$
	\mathcal{D} :  \End(C^\bullet) \longrightarrow \End(C^{\bullet + 1}),
	 \ \ \mathcal{D}(f) = [\partial,f].
$$

Let $H(\End(C^\bullet))$ be the homology of $\End(C^\bullet)$.  
Then there is a canonical morphism of Lie superalgebras
$$
	H(\End(C^\bullet)) \longrightarrow \End(H(C^\bullet)).
$$

Explicitly, if $\theta \in \End({C^\bullet})$ and $x \in C^\bullet$ are cycles representing homology classes $[\theta]$ and $[x]$, resp., then one makes the well-defined assignment \[[\theta]([x]) := [\theta(x)].\]

Note that the cycles in $(\End^n(C^\bullet),\mathcal{D})$ are precisely the chain maps (or skew-chain maps, depending on the parity of $n$): \[C^\bullet \rightarrow C^\bullet[n],\footnote{Here ``[n]" is the height shift operator on a chain complex: $C[n]^\bullet := C^{\bullet + n}$.}\] and the boundaries are precisely those chain maps that are chain homotopic to $0$. Informally, one views the image of $H(\End(C^\bullet))$ under the canonical morphism above as the collection of (graded) chain maps on $C^\bullet$, modulo homotopy.

\subsection{The Lie superalgebra  $\exsltwo_{dg}$}\label{sec:exsltwodg}


We now describe a $\Z$--graded Lie superalgebra $\exsltwo_{dg}$ that is closely related (cf. Lemma \ref{lem:homcurrent}) to the Lie superalgebra $\exsltwo$ defined in Section \ref{sec:exsltwo}.  

The underlying $\Z$--graded super vector space of $\exsltwo_{dg}$ has degree $0$ generators $\{e,f,h\}$ and degree $1$ generators $\{v_2,v_{-2},d,D\}$ with defining super commutation relations:
\begin{itemize}
\item $[e,f] = h$;
\item $[h,e] = 2e$;
\item $[h,f] = -2f$;
\item $[e,v_{2}] = 0$;
\item $[e,v_{-2}] = -[f,v_2]$;
\item $[f,v_{-2}]= 0$;

\item $[h,v_2]  = 2v_2$;
\item $[h,v_{-2}] = -2v_{-2}$;
\item $[d,y] = 0$ for all $y\in \{e,f,h,v_2,v_{-2}\}$;
\item $[D,y] = 0$ for all $y\in \{e,f,h,v_2,v_{-2}\}$;
\item $[d,d] = [D,D] = [v_2,v_2] = [v_{-2},v_{-2}] = 0$.
\item $[v_2,v_{-2}] + [d,D] = 0$.
\end{itemize}

Let $\tilde{v_0} = [e,v_{-2}] = -[f,v_2]$, and let $x = [v_2,v_{-2}] = -[d,D] = \frac{1}{2}[\tilde{v}_0,\tilde{v}_0]$.  Then we have the following.


\begin{lemma}\label{lem:basis}
The set $\{e,f,h,v_2,v_{-2},\tilde{v_0},d,D,x\}$ is a basis of $\exsltwo_{dg}$.
\end{lemma}
\begin{proof}
An easy computation shows that these elements span $\exsltwo_{dg}$.  Their linear independence follows from the representation on the annular chain complex constructed in Section \ref{sec:ckhsl2}.
\end{proof}

It is clear from the above description that $\exsltwo_{dg}$ is closely related to $\exsltwo$. 
Indeed, we may regard $\exsltwo_{dg}$ as a chain complex by declaring the adjoint action of $d$ to be the differential, as described in the previous subsection. We have the following.

\begin{lemma}\label{lem:homcurrent}
The homology of $\exsltwo_{dg}$ taken with respect to the differential $[d,\cdot]$ is isomorphic to the direct sum of $\exsltwo$ and the trivial Lie super algebra:
$$
H(\exsltwo_{dg}, [d,\cdot]) \cong \exsltwo \oplus \mathbb{C}.$$
Moreover, if $(C^\bullet, d)$ is a $\Z$--graded $\exsltwo_{dg}$ representation, regarded as a chain complex with differential given by the action of $d$, then the canonical map \[H(\exsltwo_{dg}) \rightarrow \End(H(C^\bullet,d))\] factors through $\exsltwo$.
\end{lemma}

\begin{proof}
Referring to the bracket relations and Lemma \ref{lem:basis}, we see that the kernel of $[d,\cdot]$ is spanned by $\{e,f,h,v_2,\tilde{v}_0,v_{-2},d,x\}$, while the image is spanned by $x$.
The obvious map $H(\exsltwo_{dg},[d,\cdot]) \rightarrow \exsltwo$ which takes $e,f,h$ to $e,f,h$, takes $v_2,v_{-2}$ to $v_2,v_{-2}$, $\tilde{v}_0$ to $v_0$, and $d$ to $0$ is therefore surjective, with $1$--dimensional kernel. Moreover, if $C^\bullet$ is any $\exsltwo_{dg}$--representation, then $[d] \in H(\exsltwo_{dg},[d,\cdot])$ will act trivially on $H(C^\bullet,d)$.
\end{proof}


\subsection{The current algebra $\exsltwo$ and its action on $\SKh(\bL)$}

We are now ready to extend the action of $\sltwo$ on $\CKh(\bL)$ to an action
of $\exsltwo_{dg}$ by defining
$$
	\Phi: \exsltwo_{dg} \longrightarrow \End(\CKh(\bL))
$$ via
\begin{itemize}
\item $v_2 \mapsto \partial^{Lee}_+$,
\item $v_{-2} \mapsto \partial_-$,
\item $d \mapsto \partial_0$,
\item $D \mapsto \partial^{Lee}_0$.
\end{itemize}
	
\begin{proposition}\label{prop:dgcurrentaction}
The above assignment defines a homomorphism of Lie superalgebras
$$
	\Phi : \exsltwo_{dg} \longrightarrow \End(\CKh(\bL)).
$$
\end{proposition}
\begin{proof}
The $\sltwo$ relations involving only $\{e,f,h\}$ were established in the course of proving Proposition \ref{prop:homology}.  The relations involving commutators between pairs of degree one elements $\{v_{-2},v_2,d,D\}$ follow immediately from the relations in Lemma \ref{lem:partials}.  The mixed relations involving commutators of one of $\{e,f,h\}$ with one of $\{v_{-2},v_2,d,D\}$ follow from a  straightforward case-by-case check (which we omit) along the three possible types of edges in the cube of resolutions.
\end{proof}

Passing to homology and using Lemma \ref{lem:homcurrent}, we arrive at the following proposition, which completes the proof of Theorem
\ref{thm:extcurrsl2} from the introduction.

\begin{proposition}\label{prop:current}
The homomorphism of Lie superalgebras $\Phi:\exsltwo_{dg}\longrightarrow \End(\CKh(\bL))$ induces a homomorphism of Lie superalgebras
$$
	\Psi: \exsltwo \longrightarrow \End(\SKh(\bL)).
$$
This action of $\exsltwo$ is functorial for annular link cobordisms.
\end{proposition}
\begin{proof}
Since the generator $d$ of $\exsltwo_{dg}$ is sent to the annular differential $\partial_0$, the homology of $\exsltwo_{dg}$ taken with respect to $d$ acts on $\SKh(\bL)$.  By Lemma \ref{lem:homcurrent}, the action of this homology factors through the current algebra $\exsltwo$.

What remains is to show that any annular link cobordism commutes (up to homotopy) with the operators $v_2,v_{-2}$.
To see this, let $m$ be the chain map induced on the \emph{ordinary} Khovanov chain complex by an annular link cobordism.  Write
$m=  m_0 + m_-$, where $m_0$ preserves the annular $k$-grading and $m_-$ has $k$-degree $-2$.  The claim is that $m_0$ and $v_{-2} (= \partial_-)$ commute up to homotopy.
Since $m$ is a chain map, Khovanov's differential $\partial = \partial_0 + \partial_-$ commutes with $m$, and it follows that 
$$
	m_0\partial_-  - \partial_-m_0 + m_- \partial_0 - \partial_0m_- = 0.
$$
Thus $m_-$ provides a homotopy between $[m_0,\partial_-]$ and $0$, as desired.  The analogous statements for $v_2$ follows from the above, combined with the observations (see Lemmas \ref{lem:inv} and \ref{lem:partials}) that $\partial_+^{Lee} = \Theta\partial_-\Theta$ and $\Theta$ commutes with $m_0$.
\end{proof}


\begin{remark}
A curious point to note in the proof of Proposition \ref{prop:current} is that most of the proof works equally well if we use the annular Lee deformation $\partial^{Lee}_0$ in place of the usual annular differential $\partial_0$: the commutation relations that held on the nose at the chain level still hold, and the differential/homotopy roles of $\partial_0$ and $\partial^{Lee}_0$ are simply exchanged.  (This is the symmetry between $d$ and $D$ in $\exsltwo_{dg}$.)  Indeed, one quickly checks that Lemma \ref{lem:homcurrent} holds equally well with ``$D$" replacing ``$d$" everywhere in the statement and the proof. Thus the current algebra $\exsltwo$ also acts on the homology of $\CKh(\cP(\bL))$ taken with respect to the differential $\partial^{Lee}_0$.  On the other hand, the homology with respect to $\partial^{Lee}_0$ is neither functorial for annular link cobordisms nor is it an annular link invariant.
\end{remark}

The observant reader may now wonder whether the $\exsltwo_{dg}$--action on the sutured annular chain complex $\CKh(\bL)$ gives rise to any interesting new actions on Khovanov or Lee homology. Sadly, the answer is no, as we see in Lemmas \ref{lem:homKhcurrent} and \ref{lem:homLeecurrent}. In what follows, let $\mathcal{L}$ denote the $2$--dimensional abelian Lie superalgebra with a single degree $0$ generator, $y_0$, and a single degree $1$ generator, $y_1$. 

\begin{lemma} \label{lem:homKhcurrent}
The homology $H(\exsltwo_{dg}, [d+v_{-2},\cdot])$ has a codimension $1$ direct summand isomorphic to $\mathcal{L}$. Moreover, if $(C^\bullet, d+v_{-2})$ is any $\Z$--graded $\exsltwo_{dg}$ representation, regarded as a chain complex with differential $d+v_{-2}$, then the canonical map \[H(\exsltwo_{dg}) \rightarrow \End(H(C^\bullet,d+v_{-2}))\] factors through $\mathcal{L}$.
\end{lemma}

\begin{proof} The set $\{[f],[v_2+D],[d]\}$ is a basis for $H(\exsltwo_{dg}, [d+v_{-2},\cdot])$, and we calculate that all pairwise brackets are nullhomologous. The map \[H(\exsltwo_{dg},[d+v_{-2},\cdot]) \rightarrow \mathcal{L}\] sending $[f] \mapsto y_0$, $[v_2+D] \mapsto y_1$, and $[d] \mapsto 0$ is a Lie superalgebra homomorphism, and $[d] = [d+v_{-2}]$ will act trivially on the homology of any $\exsltwo_{dg}$--representation.
\end{proof}

In particular, the action of $\exsltwo_{dg}$ on the annular chain complex $\CKh(\bL)$ induces two commuting endomorphisms of $\Kh(\bL)$, one of which can be described in terms of the standard action by basepoints on Khovanov homology (compare Section \ref{sec:markpt}) and the other of which can be described in terms of the Lee deformation.

\begin{lemma} \label{lem:homLeecurrent}
The homology $H(\exsltwo_{dg}, [d+v_{-2}+D+v_2,\cdot])$ has a codimension $1$ direct summand isomorphic to $\mathcal{L}$. Moreover, if $(C^\bullet, d+v_{-2}+D+v_2)$ is any $\Z$--graded $\exsltwo_{dg}$ representation, regarded as a chain complex with differential $d+v_{-2}+D+v_2$, then the canonical map \[H(\exsltwo_{dg}) \rightarrow \End(H(C^\bullet,d+v_{-2}+D+v_2))\] factors through $\mathcal{L}$.
\end{lemma}

\begin{proof} Similar. In this case, the set $\{[e+f],[d+v_{-2}],[d+v_{-2} + D+v_2]\}$ is a basis for the homology.
\end{proof}

Thus the action of $\exsltwo_{dg}$ on the annular chain complex $\CKh(\bL)$ induces two commuting endomorphisms of the Lee homology of $\bL$. It is straightforward to verify (e.g., by using the canonical generators described in \cite[Sec. 2.4]{Rasmussen_Slice}) that the endomorphism represented by $[e+f]$ acts by scalar multiplication on each Lee homology class associated to an orientation of $\bL$. The endomorphism  represented by $[d+v_{-2}]$ increases homological grading by $1$, hence must act trivially, since Lee homology is supported in even homological gradings \cite[Prop. 4.3]{Lee}.


\section{The symmetric group and the homology of cables}\label{sec:Sn}
In this section we explore the further symmetry exhibited by the sutured annular Khovanov homology of a cable.

In what follows, let $TL_n(1)$ denote the endomorphism algebra\footnote{The notation emphasizes that we are taking the $q=1$ specialization of $TL_n(q)$, the endomorphism algebra of the $n$th tensor product of the defining $U_q(\sltwo)$ module.} of the $n$th tensor product of the defining $\sltwo$ representation:
$$
	TL_n(1) := \End_{U(\sltwo)}(V^{\otimes n}).
$$

The standard presentation of $TL_n(1)$ has generators $\{e_i\}_{i=1}^{n-1}$ and relations
$$
	e_i^2 = -2e_i, \ e_ie_j = e_j e_i \text{ for } i\neq j\pm1, \ e_i e_{i\pm1} e_i = e_i.
$$

The goal of this section is to establish the following theorem.
\begin{theoremSnAction}\label{thm:symmetricrepresentation}
Let $K \subset S^3$ be a knot, and let $\bL = K_{n,nm} \subset A \times I$ denote its $m$--framed $n$--cable. Then there is an action of the symmetric group $\mathfrak{S}_n$ on the sutured annular Khovanov homology, $\SKh(\bL)$.  This $\mathfrak{S}_n$-action enjoys the following properties:
\begin{enumerate}
\item
it commutes with the $\exsltwo$-action;
\item
it preserves the $(i,j',k)$ tri-grading on $\SKh(\bL)$.
\item
it is natural with respect to smooth annular framed link cobordisms;
\item
it factors through the Temperley-Lieb algebra $TL_n(1)$.
\end{enumerate}
\end{theoremSnAction}

Since the $\fS_n$ action will be defined using annular link cobordisms, the fact that such an action commutes with the $\exsltwo$ action and preserves the trigrading is immediate.  Thus the rest of this section will be devoted to establishing the last two claims.

\subsection{Cobordism maps associated to tangles} Let $K$ be a smooth framed oriented knot in $A\times I$ and
\[
\iota\colon S^1\times D^2\longrightarrow A\times I
\]
an imbedding which sends the circles $S^1\times\{0\}$ and $S^1\times\{1\}$ respectively to $K$ and to a longitude of $K$ specifying the framing, where $D^2$ is the closed unit disk in $\mathbb{C}$ and $S^1:=\partial D^2$. Moreover, let $P_n\subset\operatorname{int}(D^2)$ be a collection of $n$ evenly spaced points on the real axis. In this situation, the $n$-cable of $K$ can be defined as follows:

\begin{definition} The \emph{$n$-cable} of $K$ is the $n$-component link $K^n:=\iota(S^1\times P_n)$.
\end{definition}

Now suppose $T$ is a smooth $(n,n)$ tangle, i.e., a smooth properly imbedded $1$-manifold $T\subset D^2\times I$ such that $\partial T=P_n\times\partial I$. To $T$, we can associate the ``surface of revolution'' $S^1\times T\subset S^1\times D^2\times I$.

\begin{definition}\label{def:cable_cobordism}
The \emph{$T$-cable cobordism of $K$} is the link cobordism $K^T\colon K^n\rightarrow K^n$ defined by $K^T:=\overline{\iota}(S^1\times T)$ where $\overline{\iota}\colon S^1\times D^2\times I\rightarrow A\times I\times I$ denotes the imbedding given by $\overline{\iota}(\theta,z,t):=(\iota(\theta,z),t)$ (see Figure~\ref{fig:st}).
\end{definition}

\begin{remark} If $T$ represents the elementary braid group generator $\sigma_i$, then $K^T$ can alternatively be described as the link cobordism traced out by an isotopy of $K^n$ which moves the $i$th strand of $K^n$ around the $(i+1)$st strand, thereby exchanging these two strands.
\end{remark}

\begin{figure}
\includegraphics[height=2in]{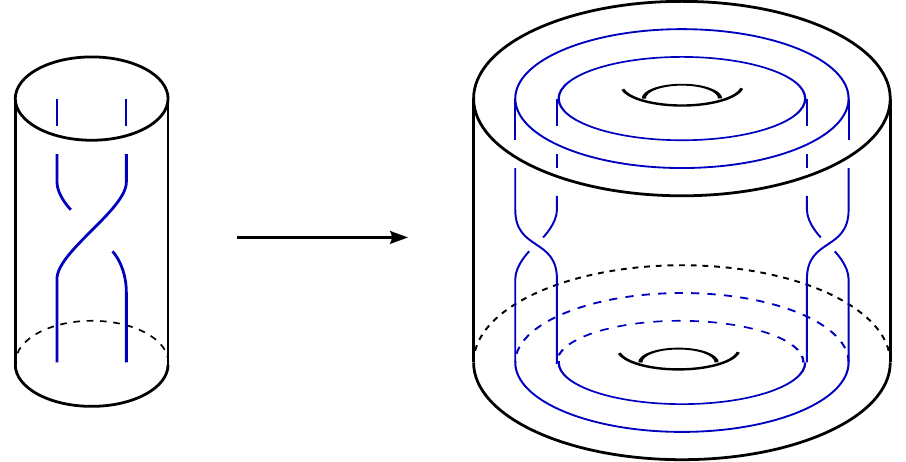}
\caption{\label{fig:st}Schematic depiction of a tangle $T\subset D^2\times I$ (left) and the associated surface of revolution $S^1\times T\subset S^1\times D^2\times I$ (right). The cobordism $K^T\colon K^n\rightarrow K^n$ is obtained from $S^1\times T$ by applying the imbedding $\overline{\iota}\colon S^1\times D^2\times I\rightarrow A\times I\times I$.}
\end{figure}

Since the formal Khovanov bracket of annular links is functorial with respect to smooth annular link cobordisms (by Proposition~\ref{prop:functorial}), the $T$-cable cobordism of $K$ induces a chain map
\[
[K^T]\colon [K^n]\longrightarrow[K^n],
\]
which is well-defined up to sign and homotopy. Likewise, $K^T$ induces maps
\[
\phi_{K^T}\colon \SKh(K^n)\longrightarrow \SKh(K^n)\quad\mbox{and}\quad
\phi^\prime_{K^T}\colon \Kh^\prime(K^n)\longrightarrow \Kh^\prime(K^n)\,,
\]
where $\Kh^\prime(K^n)$ denotes the Lee homology of $K^n$ (viewed as a link in $\mathbb{R}^3$). Note that the latter maps are well-defined up to sign.


\subsection{Fixing the sign ambiguity}\label{subs:signconvention} In the following, let $E_i$, $\Sigma_i$, and $\Sigma_i^{-1}$ denote $(n,n)$ tangles which represent respectively the generator $e_i$ of the Temperley-Lieb algebra $TL_n(a)$, the generator $\sigma_i$ of the braid group $\mathfrak{B}_n$, and the inverse of the generator $\sigma_i\in\mathfrak{B}_n$ (see Figure~\ref{fig:generators}).

\begin{figure}
\centerline{
\includegraphics[height=0.7in]{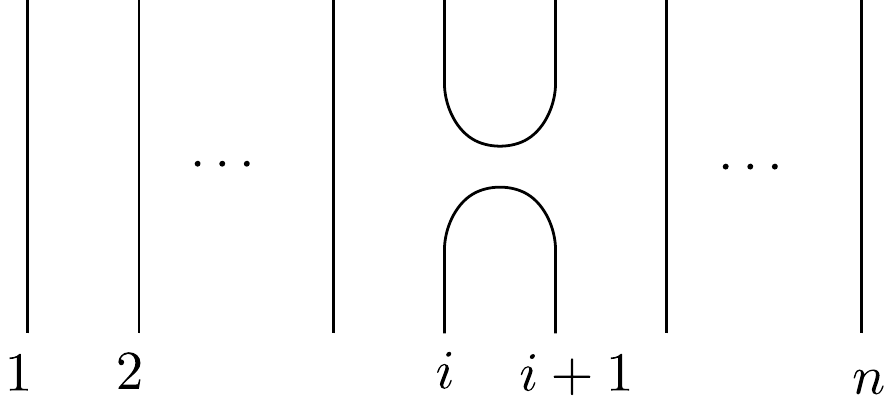}\hspace*{0.8in}
\includegraphics[height=0.7in]{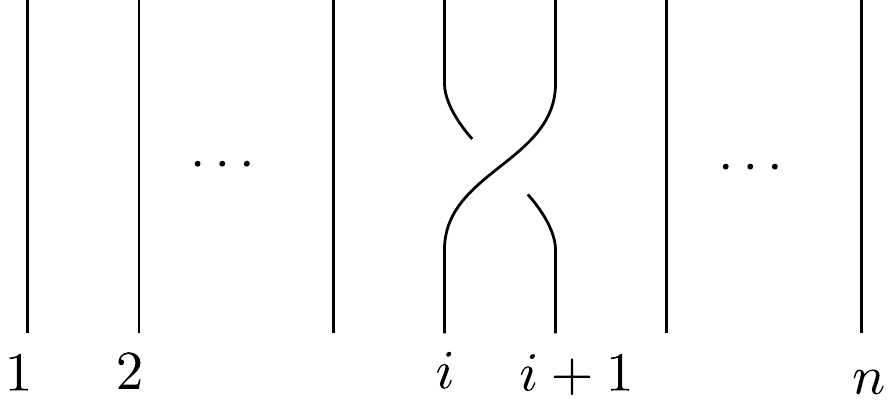}
}
\caption{\label{fig:generators}The tangles $E_i$ (left) and $\Sigma_i$ (right).}
\end{figure}

The goal of this subsection is to pin down the sign in the definition of the maps $[K^T]$, $\phi_{K^T}$, and $\phi^\prime_{K^T}$, for the case where $T$ is one of the above tangles. We will need the following theorem, due to Rasmussen:

\begin{theorem}[Rasmussen {\cite[Prop. 3.2]{Rasmussen_Surfaces}}]\label{thm:rasmussen}
The Lee homology group $\Kh^\prime(L)$ has a canonical basis whose vectors correspond bijectively to possible orientations on $L$. Furthermore, if $S\colon L\rightarrow L^\prime$ is a smooth link cobordism with no closed components, then the matrix entries of $S$ relative to the canonical bases of $\Kh^\prime(L)$ and $\Kh^\prime(L^\prime)$ satisfy
\[
(\phi^\prime_S)_{o^\prime o}=2^{-\chi(S)}\begin{cases}
\epsilon_{o^\prime o}&\mbox{if $o\cup o^\prime$ extends over $S$,}\\
0&\mbox{else,}
\end{cases}
\]
for any two orientations $o$ and $o^\prime$ on $L$ and $L^\prime$, where $\epsilon_{o^\prime o}\in\{\pm 1\}$.
\end{theorem}

\begin{remark} The canonical basis vectors referred to in this theorem are rescaled versions of the basis vectors introduced by Lee in \cite{Lee} and used by Rasmussen in \cite{Rasmussen_Slice}.
\end{remark}

The above theorem implies that if $T=\Sigma_i$ or $T=\Sigma_i^{-1}$, then
\[
(\phi^\prime_{K^T})_{o_po_p}\in\{\pm 1\}
\]
where  $o_p$ denotes the parallel orientation of $K^n$, i.e., the orientation for which all strands of $K^n$ are oriented parallel to the orientation of $K$. Likewise, the theorem implies that if $K=E_i$, then
\[
(\phi^\prime_{K^T})_{o_ao_a}\in\{\pm 1\}
\]
where $o_a$ denotes the alternating orientation of $K^n$, i.e., the orientation for which the strands of $K^n$ are alternatingly oriented parallel and antiparallel to $K$, in such a way that the leftmost strand of $K^n$ (corresponding to the leftmost point of $P_n\subset\operatorname{int}(D^2)\cap\mathbb{R}$) is oriented parallel to $K$.

Now note that \[\phi_{K^T}=\mathcal{F}([K^T])\quad\mbox{and}\quad \phi^\prime_{K^T}=\mathcal{F}^\prime([K^T])\,,\]
where $\mathcal{F}$ is the functor defined in Section~\ref{sec:ckhsl2} and $\mathcal{F}^\prime$ denotes Lee's TQFT \cite{Rasmussen_Slice}. Since $\mathcal{F}$ and $\mathcal{F}^\prime$ are additive functors, it follows that any sign choice for $[K^T]$ induces corresponding sign choices for $\phi_{K^T}$ and $\phi^\prime_{K^T}$, and we can therefore pin down the sign of $[K^T]$ (and hence of $\phi_{K^T}$) by imposing the following conventions:


\begin{convention}\label{conv:one}
For $T=\Sigma_i$ or $T=\Sigma_i^{-1}$, define the sign of $[K^T]$ to be such that the corresponding map on Lee homology satisfies $(\phi^\prime_{K^T})_{o_po_p}=+1$.
\end{convention}

\begin{convention}\label{conv:two}
For $T=E_i$, define the sign of $[K^T]$ to be such that the corresponding map on Lee homology satisfies $(\phi^\prime_{K^T})_{o_ao_a}=-1$.
\end{convention}

Convention~\ref{conv:one} immediately implies:

\begin{proposition}\label{prop:braidrepresentation}
The assignments $\sigma_i\mapsto [K^{\Sigma_i}]$ and $\sigma_i^{-1}\mapsto [K^{\Sigma_i^{-1}}]$ define a representation $\rho\colon\mathfrak{B}_n\rightarrow\operatorname{End}_{\mathcal{C}}([K^n])$ of the braid group $\mathfrak{B}_n$ where $\mathcal{C}$ denotes the bounded homotopy category of Bar-Natan's cobordism category $\operatorname{Mat}(\mathcal{C}\mbox{ob}^3_{/\ell}(A))$.
\end{proposition}
\begin{proof}
Since the cobordism maps induced on the formal Khovanov bracket are isotopy invariants when considered up to sign and homotopy, it is clear that the maps $[K^{\Sigma_i}],[K^{\Sigma_i^{-1}}]\in\operatorname{End}_{\mathcal{C}}([K^n])$ satisfy the braid group relations up to possible signs. Moreover, Rasmussen's theorem together with Convention~\ref{conv:one} implies
\[
(\phi^\prime_{K^{\Sigma_i}})_{o o_p}=(\phi^\prime_{K^{\Sigma_i^{-1}}})_{o o_p}=\begin{cases}1&o=o_p,\\
0&o\neq o_p,
\end{cases}
\]
and so the canonical basis vector associated to the orientation $o_p$ is an eigenvector for $\phi^\prime_{K^{\Sigma_i}}$ and for $\phi^\prime_{K^{\Sigma_i^{-1}}}$ for the eigenvalue $1$.  Thus, the maps  $\phi^\prime_{K^{\Sigma_i}}$ and for $\phi^\prime_{K^{\Sigma_i^{-1}}}$ trivially satisfy the braid group relations when restricted to this eigenvector, and consequently it follows that the maps $[K^{\Sigma_i}]$ and $[K^{\Sigma_i^{-1}}]$ satisfy the braid group relations on the nose (not just up to sign!).
\end{proof}

\subsection{Temperley-Lieb algebra relations for cobordism maps}
We will now show that the representation $\rho$ described in Proposition~\ref{prop:braidrepresentation} factors through the symmetric group $\mathfrak{S}_n$. This will in turn imply that the corresponding braid group action on $\SKh(K^n)$ (given by $\sigma_i\mapsto\phi_{K^{\Sigma_i}}$ and $\sigma_i^{-1}\mapsto \phi_{K^{\Sigma_i^{-1}}}$) factors through $\mathfrak{S}_n$, and thus the main statement of Theorem~\ref{thm:Sn} will follow.

Specifically, we will prove the following proposition, which holds under the assumption of Conventions~\ref{conv:one} and \ref{conv:two}, and which shows that $[K^{\Sigma_i}]$, $[K^{\Sigma_i^{-1}}]=[K^{\Sigma_i}]^{-1}$, and $[K^{E_i}]$ satisfy the \emph{symmetric group relation} $\sigma_i=\sigma_i^{-1}$ and the \emph{Kauffman bracket skein relations} $\sigma_i=a + a^{-1}e_i$ and $\sigma_i^{-1}=a^{-1}+ae_i$ at $a=1$:

\begin{proposition}\label{prop:kbrelation} The endomorphisms $[K^{\Sigma_i}], [K^{\Sigma_i^{-1}}], [K^{E_i}]\in\operatorname{End}_{\mathcal{C}}([K^n])$ satisfy:
\[
[K^{\Sigma_i}]=\operatorname{id}_{[K^n]} + [K^{E_i}]=[K^{\Sigma_i^{-1}}]\,.
\]
\end{proposition}

Instead of proving this proposition directly, we break it into two lemmas:

\begin{lemma}\label{lemma:one}
There exist signs $\epsilon_{ij}\in\{\pm 1\}$, $i=1,\ldots,n-1$, $j=1,\ldots,4$, such that
\[
[K^{\Sigma_i}]=\epsilon_{i1}\operatorname{id}_{[K^n]} + \epsilon_{i2}[K^{E_i}]\quad\mbox{and}\quad[K^{\Sigma_i^{-1}}]=\epsilon_{i3}\operatorname{id}_{[K^n]} + \epsilon_{i4}[K^{E_i}]\,.
\]
\end{lemma}

\begin{lemma}\label{lemma:two}
$\epsilon_{ij}=+1$ for all $i,j$.
\end{lemma}

The proof of Lemma~\ref{lemma:one} will be deferred to Subsection~\ref{subs:lemmaproof}.

\begin{proof}[Proof of Lemma~\ref{lemma:two}]
By Lemma~\ref{lemma:one}, we have
\[
[K^{\Sigma_i}]=\epsilon_{i1}\operatorname{id}_{[K^n]} + \epsilon_{i2}[K^{E_i}]
\]
for $\epsilon_{i1},\epsilon_{i2}\in\{\pm 1\}$, and hence the corresponding maps in Lee homology satisfy
\[
\phi^\prime_{K^{\Sigma_i}}=\epsilon_{i1}\operatorname{id} + \epsilon_{i2}\phi^\prime_{K^{E_i}}\,.
\]
By considering the matrices of the above maps relative to the basis of Theorem~\ref{thm:rasmussen} and comparing the diagonal entries corresponding to the parallel orientation $o_p$, we thus obtain \[1=\epsilon_{i1} +0\] and hence $\epsilon_{i1}=1$, where we have used Convention~\ref{conv:one} and Theorem~\ref{thm:rasmussen} to conclude that $(\phi^\prime_{K^{\Sigma_i}})_{o_po_p}=1$ and $(\phi^\prime_{K^{E_i}})_{o_po_p}=0$. Similarly, by comparing the diagonal entries corresponding to the alternating orientation $o_a$, we obtain \[0=\epsilon_{i1}-\epsilon_{i2}=1-\epsilon_{i2}\] and hence $\epsilon_{i2}=1$, where we have used Theorem~\ref{thm:rasmussen} and Convention~\ref{conv:two} to conclude that $(\phi^\prime_{K^{\Sigma_i}})_{o_ao_a}=0$ and $(\phi^\prime_{K^{E_i}})_{o_ao_a}=-1$.
The proof of $\epsilon_{i3}=\epsilon_{i4}=1$ is analogous.
\end{proof}

We can now use Proposition~\ref{prop:kbrelation} to prove the following proposition, which implies that the endomorphisms $[K^{E_i}]$ satisfy the \emph{Temperley-Lieb algebra relations}
\begin{enumerate}
\item
$e_i^2=-(a^2-a^{-2})e_i$,
\item
$e_ie_{i\pm 1}e_i=e_i$,
\item
$e_ie_j=e_je_i$ whenever $|i-j|\geq 2$,
\end{enumerate}
at $a=1$:

\begin{proposition}\label{prop:tlrelation}
The endomorphisms $[K^{E_i}]\in\operatorname{End}_{\mathcal{C}}([K^n])$ satisfy
\begin{enumerate}
\item
$[K^{E_i}]\circ[K^{E_i}]=-2[K^{E_i}]$,
\item
$[K^{E_i}]\circ[K^{E_{i\pm 1}}]\circ[K^{E_i}]=[K^{E_i}]$,
\item
$[K^{E_i}]\circ[K^{E_j}]=[K^{E_j}]\circ[K^{E_i}]=-2[K^{E_i}]$ whenever $|i-j|\geq 2$,
\end{enumerate}
where all relations hold in $\operatorname{End}_{\mathcal{C}}([K^n])$.
\end{proposition}

\begin{proof} (1) By Proposition~\ref{prop:kbrelation}, we have $[K^{E_i}]=[K^{\Sigma_i}]-\operatorname{id}_{[K^n}]$, and hence
\[
[K^{E_i}]^2=([K^{\Sigma_i}]-\operatorname{id}_{[K^n]})^2=2\operatorname{id}_{[K^n]}-2[K^{\Sigma_i}]=-2[K^{E_i}]\,,
\]
where in the second equaltity we have used that $[K^{\Sigma_i}]$ squares to $\operatorname{id}_{[K^n]}$ because $[K^{\Sigma_i}]^{-1}=[K^{\Sigma_i^{-1}}]=[K^{\Sigma_i}]$ by Propositions~\ref{prop:braidrepresentation} and \ref{prop:kbrelation}.

(2)  Since the chain maps induced by annular link cobordisms are isotopy invariants when considered up to sign and homotopy and since the tangles $E_i\circ E_{i\pm 1}\circ E_i$ and $E_i$ are isotopic, it is clear that (2) holds up to an overall sign. It is therefore clear that the induced maps in Lee homology satisfy
\[
\phi^\prime_{K^{E_i}}\circ\phi^\prime_{K^{E_{i\pm 1}}}\circ\phi^\prime_{K^{E_i}}=\epsilon\phi^\prime_{K^{E_i}}\,.
\]
for an $\epsilon\in\{\pm 1\}$.
To conclude that $\epsilon=1$, we now use Proposition~\ref{prop:kbrelation} to
write the right-hand side of the above equation as \[\epsilon(\phi^\prime_{K^{\Sigma_i}}-\operatorname{id})\] and the left-hand side as
\[
(\phi^\prime_{K^{\Sigma_i}}-\operatorname{id})\circ
(\phi^\prime_{K^{\Sigma_{i\pm 1}}}-\operatorname{id})\circ
(\phi^\prime_{K^{\Sigma_i}}-\operatorname{id})=
\phi^\prime_{K^{\Sigma_i}}\circ
\phi^\prime_{K^{\Sigma_{i\pm 1}}}\circ
\phi^\prime_{K^{\Sigma_i}}-2\operatorname{id}+\delta\,,
\]
where $\delta:=-\phi^\prime_{K^{\Sigma_i}}\circ
\phi^\prime_{K^{\Sigma_{i\pm 1}}}-\phi^\prime_{K^{\Sigma_{i\pm 1}}}\circ
\phi^\prime_{K^{\Sigma_i}}+2\phi^\prime_{K^{\Sigma_i}}+\phi^\prime_{K^{\Sigma_{i\pm 1}}}$. This yields
\[
\phi^\prime_{K^{\Sigma_i}}\circ
\phi^\prime_{K^{\Sigma_{i\pm 1}}}\circ
\phi^\prime_{K^{\Sigma_i}}-2\operatorname{id}+\delta=\epsilon(\phi^\prime_{K^{\Sigma_i}}-\operatorname{id})\,,
\]
and by considering the matrices of the above maps relative to the basis of Theorem~\ref{thm:rasmussen} and comparing the diagonal entries corresponding to the alternating orientation $o_a$, we obtain
\[
\epsilon^\prime - 2 + 0 =\epsilon(0-1)
\]
for an $\epsilon^\prime\in\{\pm 1\}$, where we have used Theorem~\ref{thm:rasmussen} to conclude that 
\[(\phi^\prime_{K^{\Sigma_i}}\circ\phi^\prime_{K^{\Sigma_{i\pm 1}}}\circ\phi^\prime_{K^{\Sigma_i}})_{o_ao_a}\in\{\pm 1\}\quad\mbox{and} \quad(\delta)_{o_ao_a}=(\phi^\prime_{K^{\Sigma_i}})_{o_ao_a}=0.\] However, since $\epsilon,\epsilon^\prime\in\{\pm 1\}$, the equation $\epsilon^\prime-2=-\epsilon$ can only hold if $\epsilon^\prime=\epsilon=1$, and hence (2) follows.

(3) Relation (3) follows because $[K^{E_i}]$ can be written as $[K^{E_i}]=[K^{\Sigma_i}]-\operatorname{id}_{[K^n]}$ (by Proposition~\ref{prop:kbrelation}) and because the $[K^{\Sigma_i}]$ satisfy the braid group relations (by Proposition~\ref{prop:braidrepresentation}).
\end{proof}

\subsection{Proof of Lemma~\ref{lemma:one}}\label{subs:lemmaproof}
In this subsection, we will assume that the knot $K\subset A\times I$ is represented by a diagram on $A$ such that the framing of $K$ is the blackboard framing. In the relevant figures, the hole of the annulus $A$ will be represented by an $\XX$. The annulus itself will not be shown.

To prove Lemma~\ref{lemma:one}, we will proceed in two steps: we will first prove the lemma in the special case where $K$ is a $0$-framed unknot, and then generalize our arguments to the case where $K$ is an arbitrary framed oriented knot in $A\times I$. We will only prove that
\[
[K^{\Sigma_i}]=\epsilon_{i1}\operatorname{id}_{[K^n]}+\epsilon_{i2}[K^{E_i}]
\]
for $\epsilon_{i1},\epsilon_{i2}\in\{\pm 1\}$, as the proof of the second equation in Lemma~\ref{lemma:one} is nearly identical.

\emph{Special case.} In the case where $K$ is a $0$-framed unknot, the cobordism $K^{\Sigma_i}$ can be represented by the movie $M^{\Sigma_i}$ shown in Figure~\ref{fig:Braid_Movie_Unknot}. The first two diagrams in this movie (henceforth denoted $D_1$ and $D_2$) differ by a Reidemeister II move, and the last two diagrams (henceforth denoted $D_3$ and $D_4$) differ by an inverse Reidemeister II move. The middle two diagrams differ by a planar isotopy which moves crossing 2 along the dashed arrow while fixing crossing 1, so that at the end of the isotopy crossing 2 comes to lie above crossing 1.

\begin{figure}
\centerline{\includegraphics[height=1.4in]{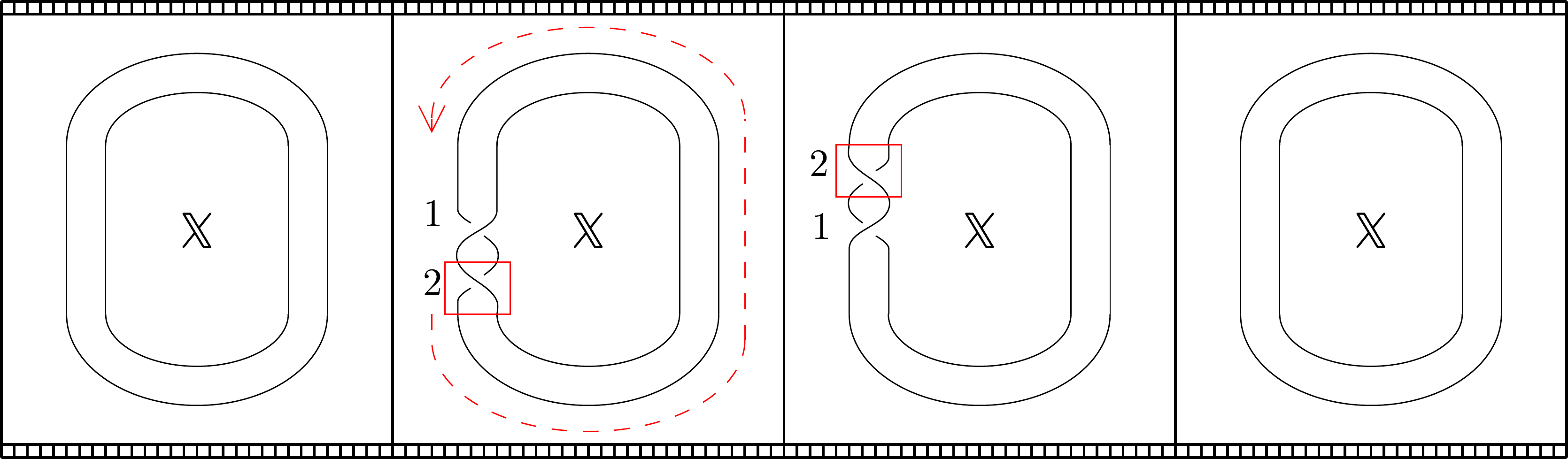}}
\caption{\label{fig:Braid_Movie_Unknot}Movie $M^{\Sigma_i}$ for the cobordism $K^{\Sigma_i}\colon K^n\rightarrow K^n$ for the case where $K$ is a $0$-framed unknot. In this figure, we have only depicted the $i$th and the $(i+1)$st strand of $K^n$, as the other strands remain unchanged over the course of the movie.}
\end{figure}

\begin{figure} \label{fig:TLCob}
\centerline{\includegraphics[height=1.4in]{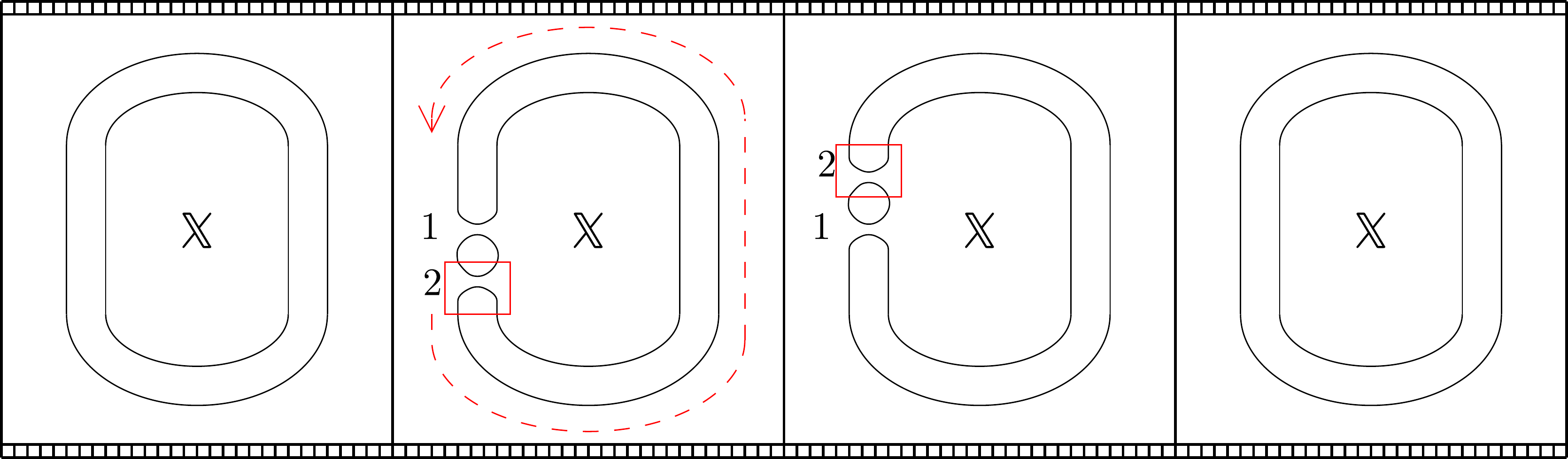}}
\caption{\label{fig:TL_Movie_Unknot}Movie $M^{E_i}$ for the cobordism $K^{E_i}\colon K^n\rightarrow K^n$ for the case where $K$ is a $0$-framed unknot.
}
\end{figure}

The chain map $[K^{\Sigma_i}]\colon[D_1]\rightarrow[D_4]$ induced by $M^{\Sigma_i}$ is thus given by
\[
[K^{\Sigma_i}]=G\circ\Psi\circ F\,,
\]
where $F$ denotes the chain map associated to the Reidemeister II move between $D_1$ and $D_2$, $\Psi$ denotes the chain map induced by the isotopy between $D_2$ and $D_3$, and $G$ denotes the chain map associated to the inverse Reidemeister II move between $D_3$ and $D_4$. Recalling the definition of $F$ and $G$ from \cite[Subs. 4.3]{MR2174270}, it thus follows that, up to possible signs, $[K^{\Sigma_i}]$ is given by the rightward pointing arrows in the following diagram:
\[
\xymatrix@R=-0.05in@C=0.63in{
[D_1]^0\ar[r]^F&[D_2]^0\ar[r]^{\Psi}&[D_3]^0\ar[r]^G&[D_4]^0\\ \\ \\ \\
&*+={\includegraphics[height=0.6in]{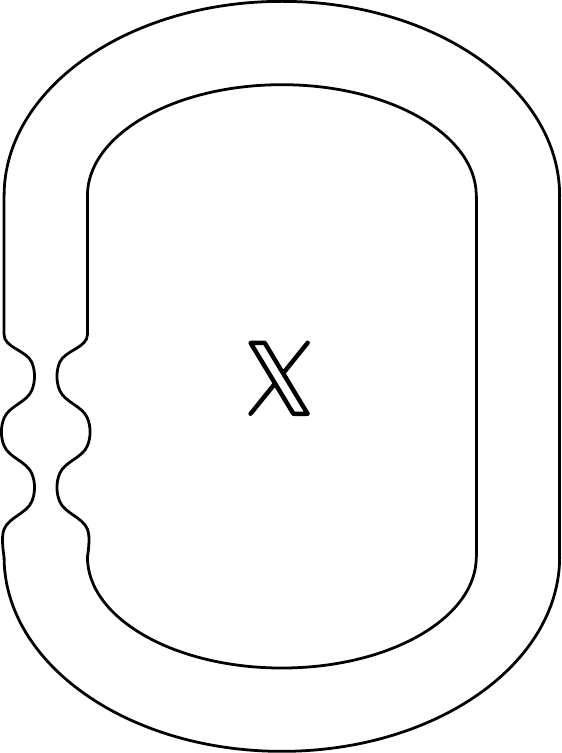}}\ar[r]^{\operatorname{id}}&*+={\includegraphics[height=0.6in]{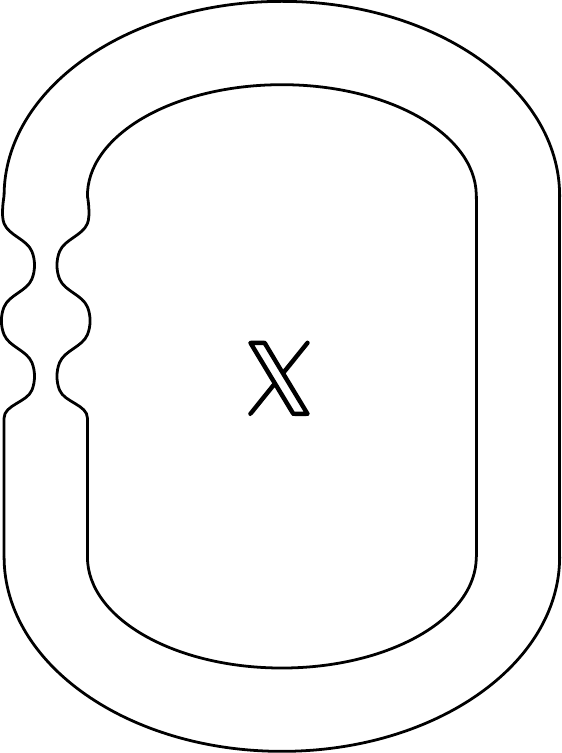}}\ar[dr]^{\operatorname{id}}&\\
*+={\includegraphics[height=0.7in]{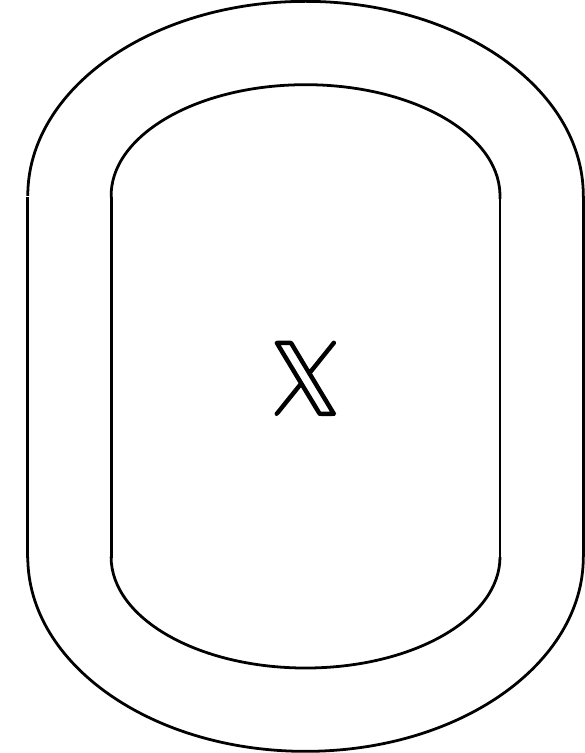}}\ar[ur]^{\operatorname{id}}\ar[dr]_f&\oplus&\oplus&*+={\includegraphics[height=0.6in]{sBraid_Movie_Chain_Map1n_X.pdf}}\\
&*+={\includegraphics[height=0.6in]{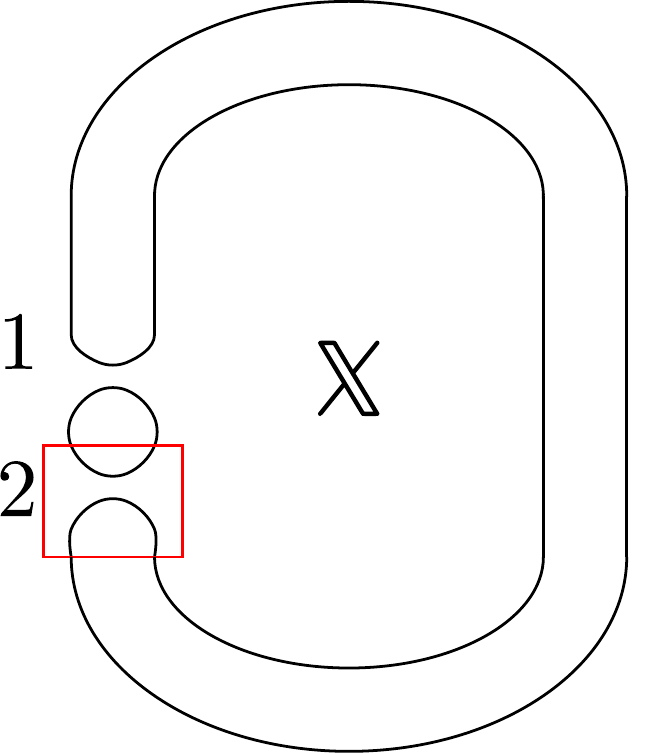}}\ar[r]_\psi&*+={\includegraphics[height=0.6in]{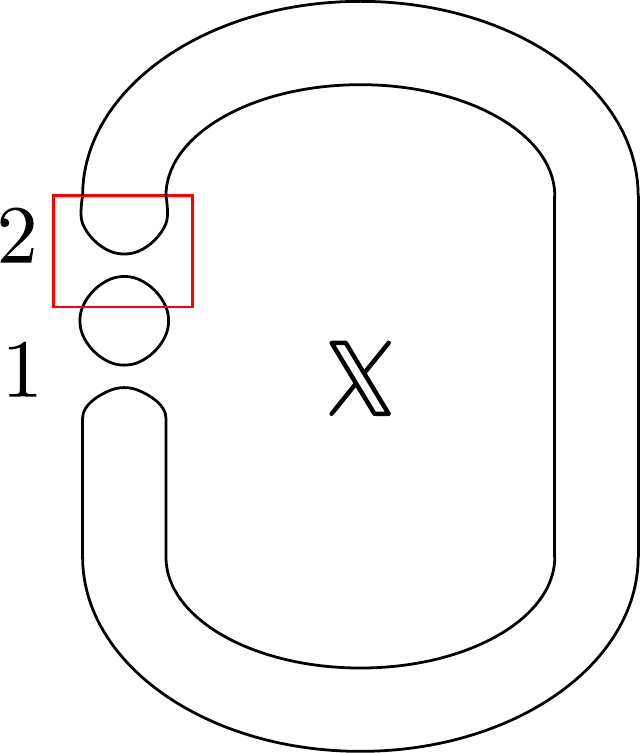}}\ar[ur]_g&
}
\]

In this diagram, the four columns represent the $0$th chain groups of the formal Khovanov brackets of $D_j$ for $j=1,\ldots,4$, and the arrows labeled $f$, $\psi$, and $g$ represent morphisms given by cobordisms in $A\times I$. It turns out that, up to possible signs, $f$, $\psi$, and $g$ are precisely the morphisms induced by moves between consecutive diagrams in the movie $M^{E_i}$ shown in Figure~\ref{fig:TL_Movie_Unknot}. In the movie $M^{E_i}$, the middle two diagrams differ by a planar isotopy which moves the interior of the red box along the dashed arrow and thereby turns the small circle in the second diagram into the elongated component in the third diagram, and vice versa. Moreover, the first two diagrams in $M^{E_i}$ differ by a saddle move followed by creation of a small circle, and the last two diagrams differ by annihilation of a small circle followed by a saddle move. (Note that there should be an intermediate diagram between the first (resp., last) two diagrams in the movie $M^{E_i}$, but we chose to suppress this intermediate diagram to make Figures~\ref{fig:Braid_Movie_Unknot} and \ref{fig:TL_Movie_Unknot} look similar).

Ignoring possible signs, it now follows from the above diagram that
\[
[K^{\Sigma_i}]=(\operatorname{id}\circ\operatorname{id}\circ\operatorname{id})+(g\circ\psi\circ f)=\operatorname{id}_{[K^n]}+[K^{E_i}]\,,
\]
as desired.

\emph{General Case.} Now suppose $K$ is an arbitrary framed oriented knot in $A\times I$. In this case, the cobordisms $K^{\Sigma_i}$ and $K^{E_i}$ can be described by the movies $M^{\Sigma_i}$ and $M^{E_i}$ shown in Figures~\ref{fig:Braid_Movie_Knot} and \ref{fig:TL_Movie_Knot}, respectively.
\begin{figure}
\centerline{\includegraphics[height=1.4in]{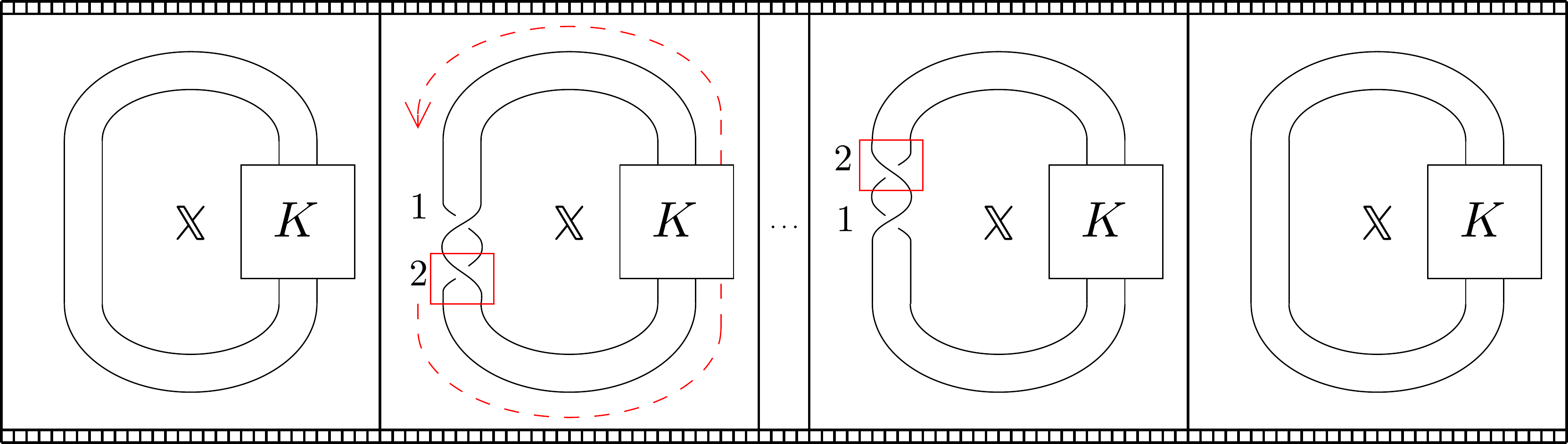}}
\caption{\label{fig:Braid_Movie_Knot}Movie $M^{\Sigma_i}$ for the cobordism $K^{\Sigma_i}\colon K^n\rightarrow K^n$.}
\end{figure}
\begin{figure}
\centerline{\includegraphics[height=1.4in]{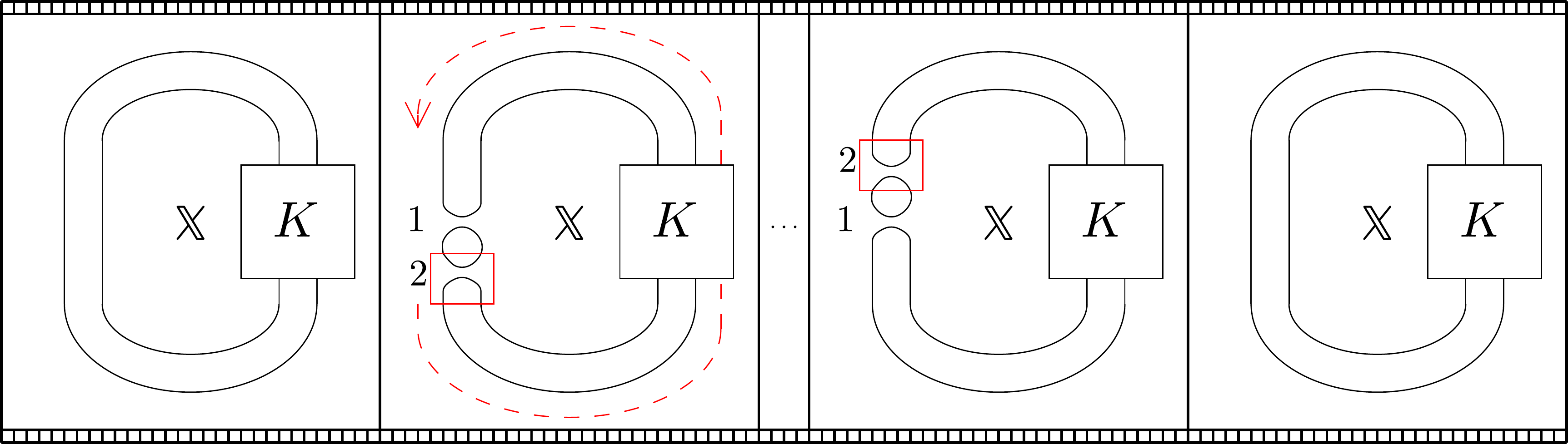}}
\caption{\label{fig:TL_Movie_Knot}Movie $M^{E_i}$ for the cobordism $K^{E_i}\colon K^n\rightarrow K^n$.}
\end{figure}
\begin{figure}
\centerline{\includegraphics[height=1.2in]{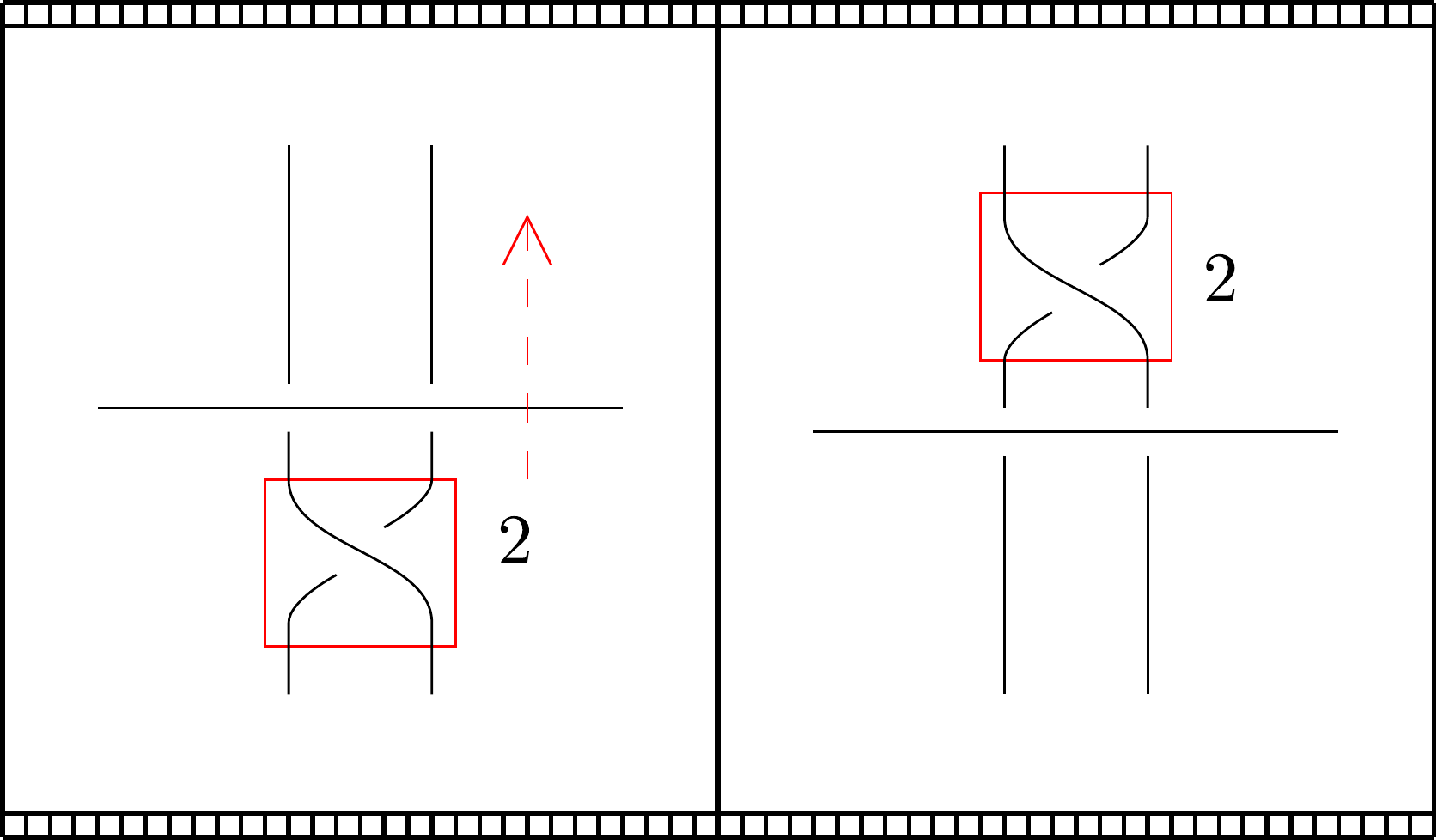}\hspace*{0.4in}\includegraphics[height=1.2in]{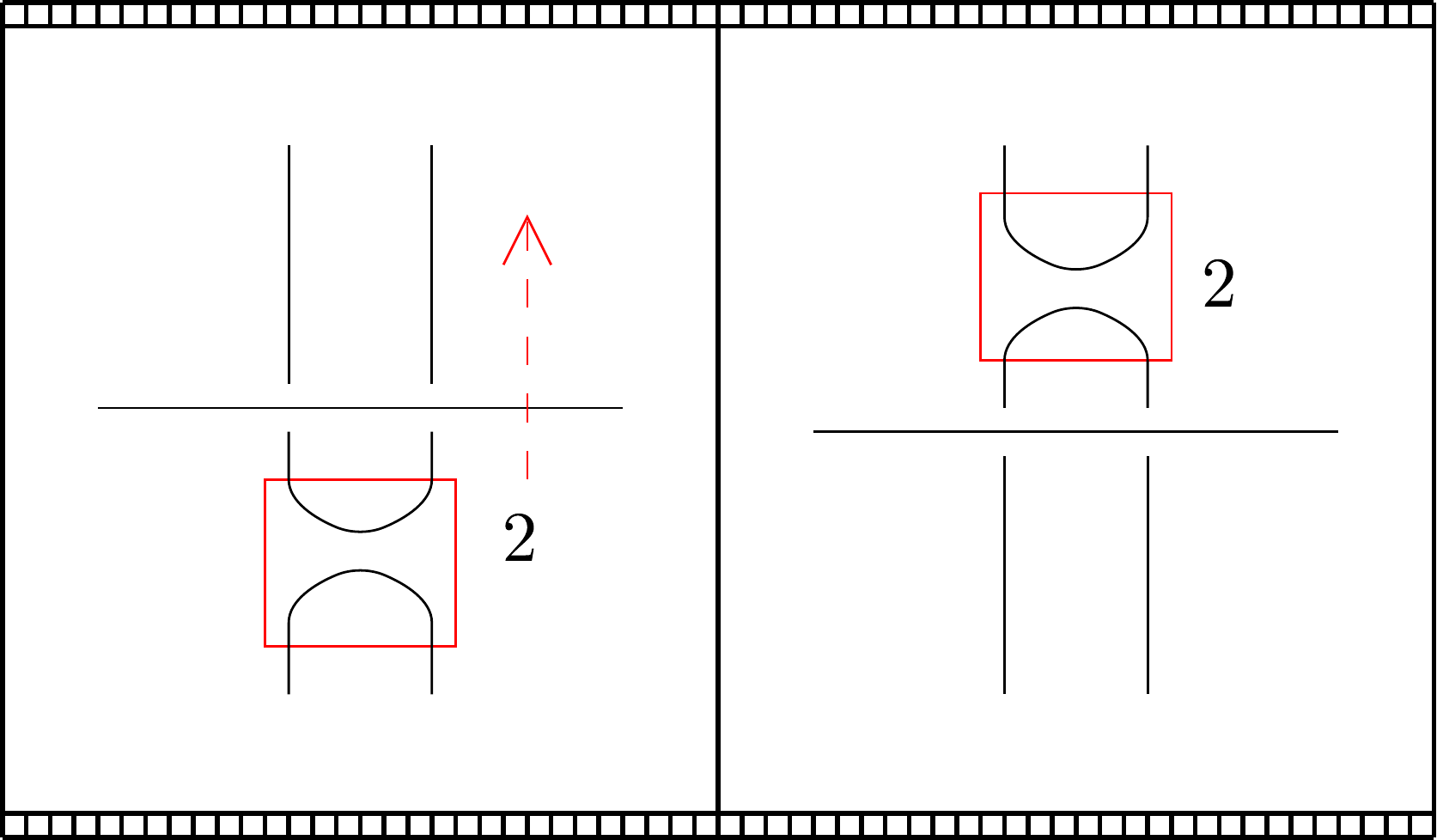}}
\caption{\label{fig:R_Moves}Sliding the red box across other strands.}
\end{figure}

Note that these movies differ from the ones in Figures~\ref{fig:Braid_Movie_Unknot} and \ref{fig:TL_Movie_Unknot} in two ways: firstly each diagram in Figures~\ref{fig:Braid_Movie_Knot} and \ref{fig:TL_Movie_Knot} contains a ``knotted'' part, which is represented by a box labeled $K$. Explicitly, this box stands for an $n$-cable diagram of a $(1,1)$ tangle whose closure is the knot $K$. Secondly, the movies in Figures~\ref{fig:Braid_Movie_Knot} and \ref{fig:TL_Movie_Knot} contain intermediate diagrams, which are represented in the figures by dots between the second and the second-to-last diagram. These intermediate diagrams arise because one has to use Reidemeister moves of type III (in the case of Figure~\ref{fig:Braid_Movie_Knot}) and type II (in the case of Figure~\ref{fig:TL_Movie_Knot}) to move the red box across possible over- and understrands located in the box labeled $K$. Local pictures of such Reidemeister moves are shown in Figure~\ref{fig:R_Moves}.

The chain map $[K^{\Sigma_i}]\colon[K^n]\rightarrow[K^n]$ associated to the movie $M^{\Sigma_i}$ is now given by
\[
[K^n]=G\circ\Psi_{\ell}\circ\ldots\circ\Psi_2\circ\Psi_1\circ F\,,
\]
where $F$ and $G$ are as in the case where $K$ is a $0$-framed unknot, and $\Psi_1,\Psi_2,\ldots,\Psi_\ell$ are the chain maps induced by the Reidemeister III moves. Comparing with \cite[Subs. 4.3]{MR2174270}, one sees that each map $\Psi_j\colon 
\left[\mbox{\raisebox{-0.12in}{\includegraphics[height=0.3in]{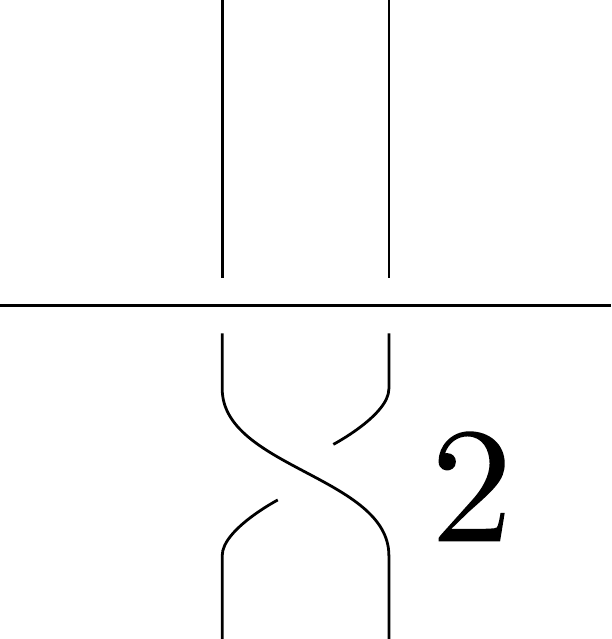}}}\right]\rightarrow
\left[\mbox{\raisebox{-0.12in}{\includegraphics[height=0.3in]{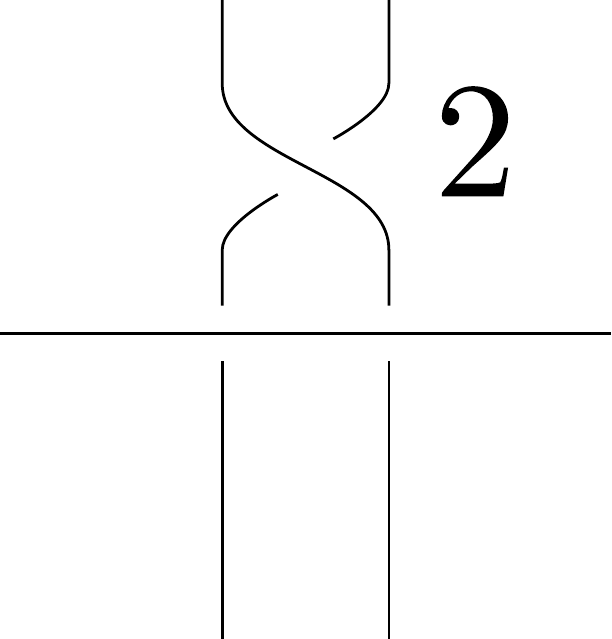}}}\right]$ has three components, which are denoted by $\operatorname{id}$, $\nu_j$, and $\psi_j$ in the following diagram:
\[
\xymatrix@C=0.63in@R=0.5in{
[\mbox{\raisebox{-0.15in}{\includegraphics[height=0.35in]{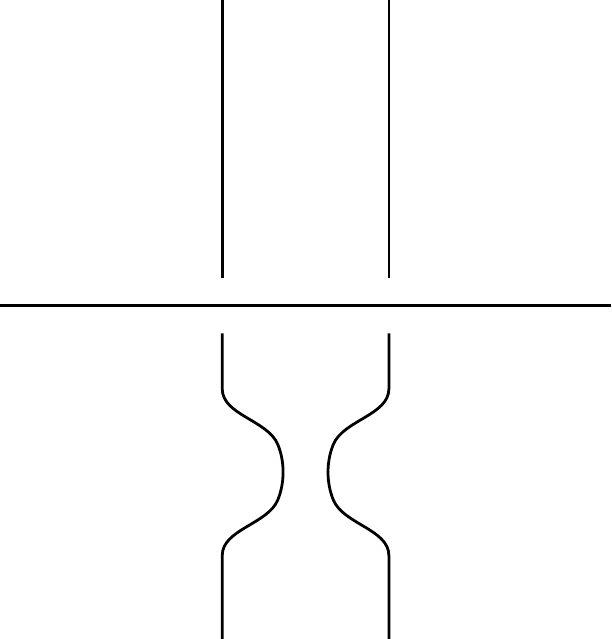}}}]\ar[r]^{\operatorname{id}} &[\mbox{\raisebox{-0.15in}{\includegraphics[height=0.35in]{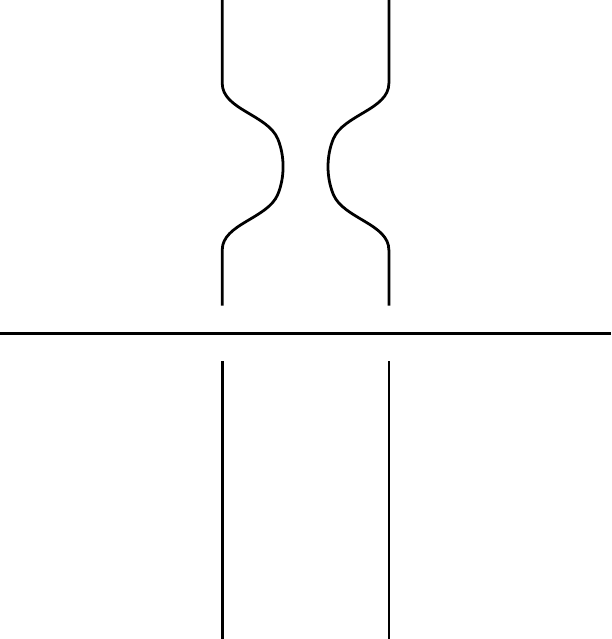}}}]\\
[\mbox{\raisebox{-0.15in}{\includegraphics[height=0.35in]{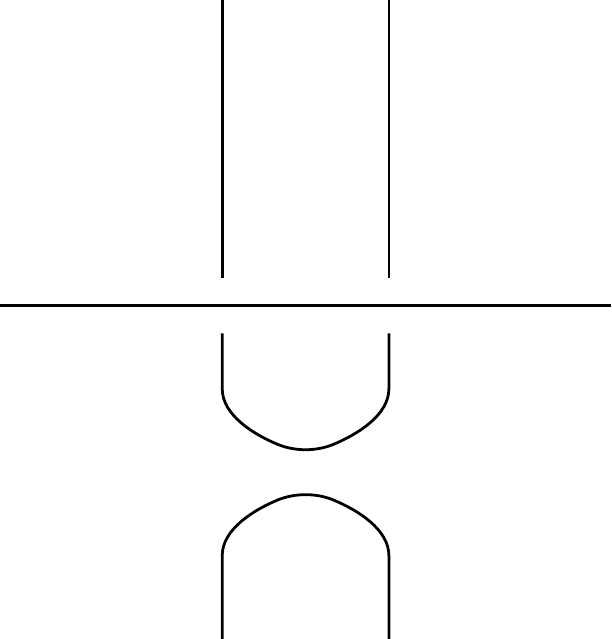}}}]\ar[r]^{\psi_j}\ar[u]^{\varphi_j}\ar[ur]^{\nu_j} &[\mbox{\raisebox{-0.15in}{\includegraphics[height=0.35in]{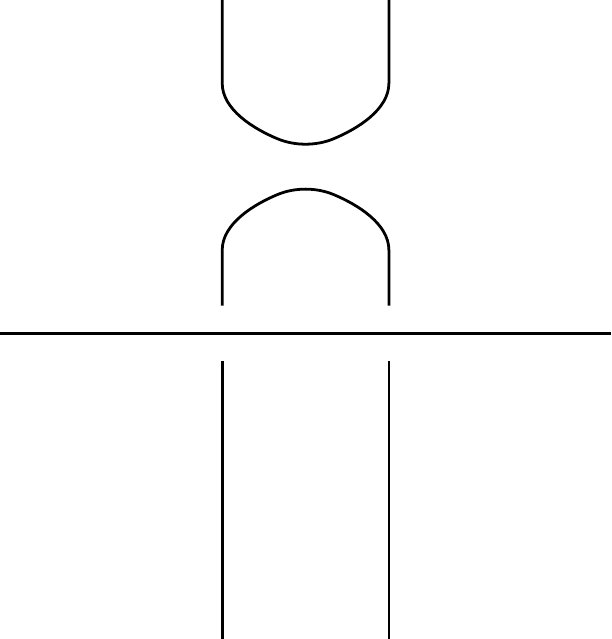}}}]\ar[u]_{\varphi_{j+1}}\\
}
\]
Here, the two columns represent the formal Khovanov brackets of $\mbox{\raisebox{-0.12in}{\includegraphics[height=0.3in]{R3_Movie_Map1.pdf}}}$ and of $\mbox{\raisebox{-0.12in}{\includegraphics[height=0.3in]{R3_Movie_Map2.pdf}}}$, written as mapping cones of chain maps $\varphi_j$ and $\varphi_{j+1}$, where $\varphi_j$ and $\varphi_{j+1}$ are induced by saddle cobordisms between the two possible resolutions of the crossing labeled 2. Moreover, $\psi_j$ is precisely the chain map induced by the movie pictured in the right half of Figure~\ref{fig:R_Moves}.

The chain map $[K^{\Sigma_i}]$ associated to the movie $M^{\Sigma_i}$ is thus given by the composition of the rightward pointing arrows in the following diagram, in which $f$ and $g$ are as in the case where $K$ is a $0$-framed unknot:
\[
\xymatrix@R=-0.05in@C=0.4in{
&*+={\includegraphics[height=0.6in]{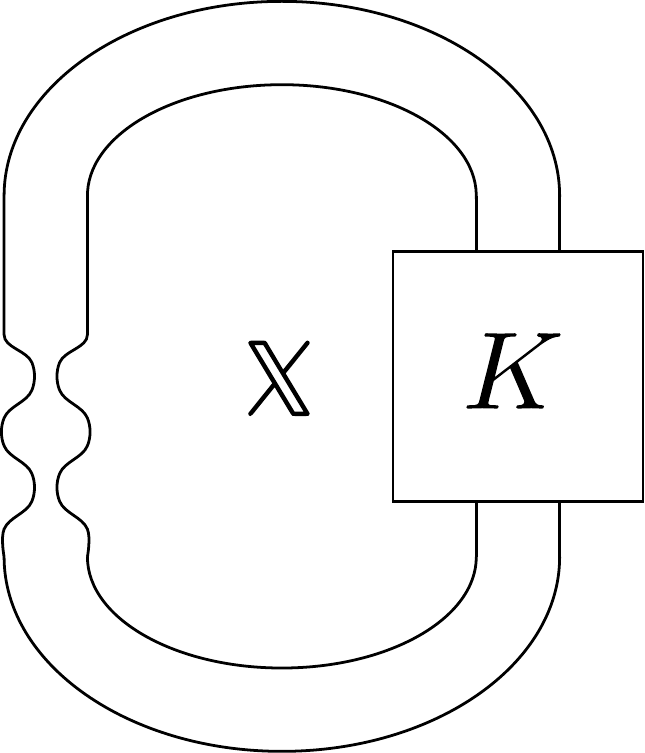}}\ar[r]^{\operatorname{id}}&*+<0.5in>{\cdots}\ar[r]^{\operatorname{id}}&*+={\includegraphics[height=0.6in]{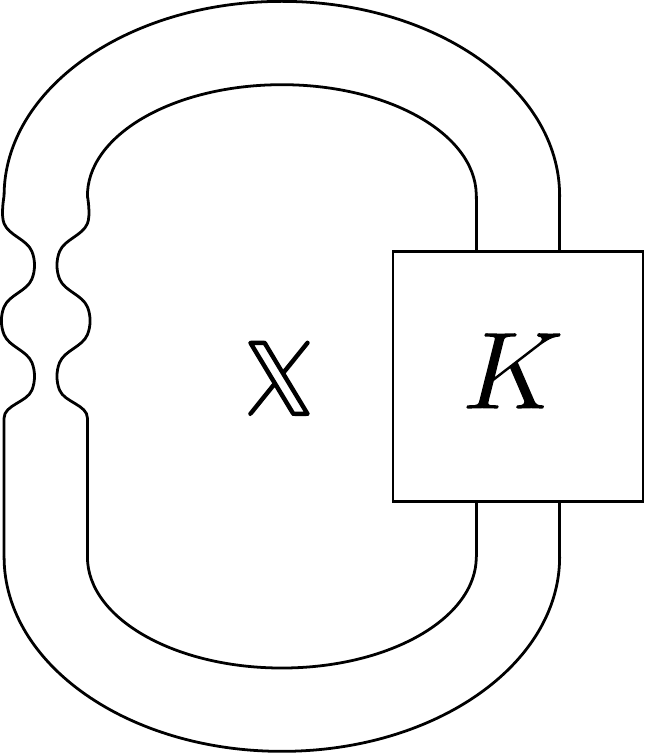}}\ar[dr]^{\operatorname{id}}&\\
*+={\includegraphics[height=0.7in]{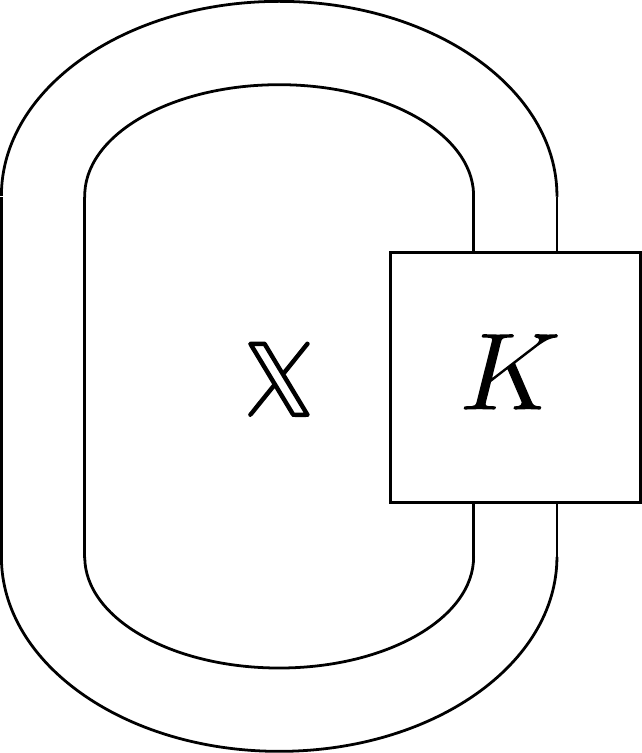}}\ar[ur]^{\operatorname{id}}\ar[dr]_{f}&\oplus&&\oplus&*+={\includegraphics[height=0.6in]{sBraid_Movie_Knot_Chain_Map1n_X.pdf}}\\
&*+={\includegraphics[height=0.6in]{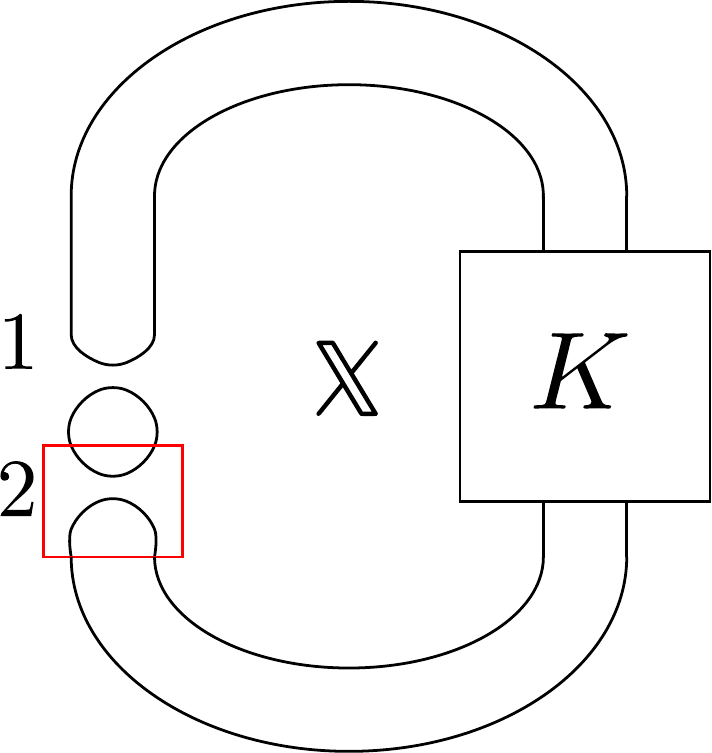}}\ar[r]_{\psi_1}\ar[uur]^{\nu_1}&*+<0.5in>{\cdots}\ar[r]_{\psi_\ell}\ar[uur]^{\nu_\ell}&*+={\includegraphics[height=0.6in]{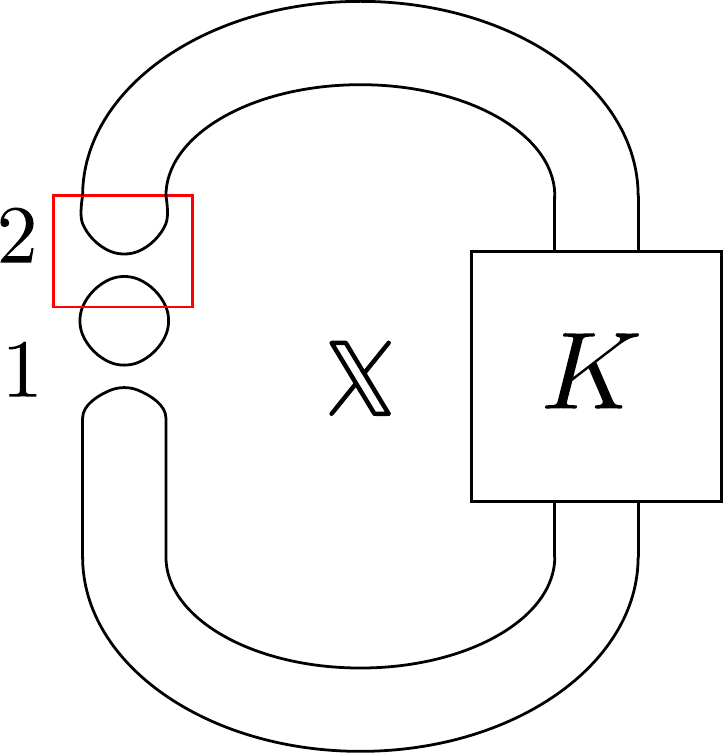}}\ar[ur]_g&
}
\]

The lemma now follows from:

\begin{sclaim} The maps $\nu_1,\nu_2,\ldots,\nu_\ell$ do not contribute to $[K^{\Sigma_i}]$. That is, up to possible signs, we have
\[
[K^{\Sigma_i}]=(\operatorname{id}\circ\ldots\circ\operatorname{id})+(g\circ\psi_\ell\circ\ldots\circ\psi_1\circ f)=\operatorname{id}_{[K^n]}+[K^{E_i}]\,,
\]
as desired.
\end{sclaim}

\begin{proof}[Proof of the claim] Let $D_1,\ldots, D_{\ell+3}$ be the link diagrams that appear in the movie $M^{\Sigma_i}$. We will say that a crossing of $D_j$ has type 1 (resp., type 2) if it is one of the crossings that were already present in $D_1$ (resp., if it is one of the two crossings labeled 1 and 2 in Figures~\ref{fig:Braid_Movie_Knot} and \ref{fig:R_Moves}). Moreover, we will regard $[D_j]$ as a bicomplex, where the first and the second differential in the bicomplex are given by all edge-maps in the resolution cube of $D_j$ which correspond to crossings of type 1 and type 2, respectively. Corresponding to the two differentials, there are two cohomological gradings, denoted $i_1$ and $i_2$, whose sum is equal to the total cohomological degree on $[D_j]$. (Explicitly, these two gradings are defined by $i_m:=k_m-n_{m-}$, where $k_m$ denotes the number of $1$-resolution at crossings of type $m$, and $n_{m-}$ denotes the number of negative crossings of type $m$, with respect to a fixed orientation for $K^n$).

Now note that each $\nu_j$ raises the $i_2$-degree by $1$ (and hence lowers the $i_1$-degree by $1$). Indeed, this follows because $\nu_j$ turns a $0$-resolution the crossing labeled 2 into a $1$-resolution while leaving the resolution at the crossing labeled 1 unchanged. (Here, we assume that crossings are labeled as in Figures~\ref{fig:Braid_Movie_Knot} and \ref{fig:R_Moves}). Moreover, it is easy to see that all other maps in the above diagram preserve the $i_1$- and the $i_2$-degree. It thus follows that the chain map $[K^{\Sigma_i}]\colon[D_1]\rightarrow[D_{\ell+3}]$ can be written as
\[
[K^{\Sigma_i}]=[K^{\Sigma_i}]_0+[K^{\Sigma_i}]_+\,,
\]
where $[K^{\Sigma_i}]_0$ preserves the $i_2$-degree, and $[K^{\Sigma_i}]_+$ strictly raises the $i_2$-degree. But since $[D_1]$ and $[D_{\ell+3}]$ are supported in $i_2$-degree $0$ (essentially be definition of the $i_2$-degree), it follows that $[K^{\Sigma_i}]_+$ has to be zero, and hence the $\nu_j$ cannot contribute to $[K^{\Sigma_i}]$ because they could only contribute to $[K^{\Sigma_i}]_+$.
\end{proof}

\begin{proposition}\label{prop:cablefunctorial}
If $S: K_1 \rightarrow K_2$ is a framed annular knot cobordism, then there is an induced homomorphism of $\fS_n$ representations 
$\SKh(K_1^{n}) \rightarrow \SKh(K_2^{n})$.
\end{proposition}

\begin{proof} Let $S:K_1 \rightarrow K_2$ be a framed knot cobordism. Then the $n$-cable of $S$ (defined by taking $n$ parallel copies of $S$) is a link cobordism $S^{n}\colon K_1^{n}\rightarrow K_2^{n}$, and hence there is an induced map
\[
\SKh(K_1^{n}) \longrightarrow \SKh(K_2^{n}).
\]
To show that this map commutes with the $\fS_n$ actions, we first note that if $\Sigma_i$ is one of the generators shown in Figure~\ref{fig:generators}, then the maps
\[
[S^{n}]\circ[K_1^{\Sigma_i}]\quad\mbox{and}\quad[K_2^{\Sigma_i}]\circ [S^{n}]
\]
agree up to an overall sign because the cobordisms
\[
S^{n}\circ K_1^{\Sigma_i}\quad\mbox{and}\quad K_2^{\Sigma_i}\circ S^{n}
\]
are isotopic.

Now let $o_p$ and $o^\prime_p$ denote the parallel orientations of $K_1^{n}$ and $K_2^{n}$ (i.e., the orientations for which all strands of the $n$-cable are oriented parallel to the orientation of the original knot). Since the orientations $o_p$ and $o^\prime_p$ are consistent with the parallel orientation of $S^{n}$, Theorem~\ref{thm:rasmussen} implies that the matrix entry $(\phi^\prime_{S^{n}})_{o^\prime_po_p}$ of the induced map in Lee homology is nonzero. Moreover, Convention~\ref{conv:one} implies that the matrix entries $(\phi^\prime_{K_1^{\Sigma_i}})_{o_po_p}$ and $(\phi^\prime_{K_2^{\Sigma_i}})_{o^\prime_po^\prime_p}$ are equal to $1$, and Theorem~\ref{thm:rasmussen} shows that all other entries in the same row and the same column of the matrices of $\phi^\prime_{K_1^{\Sigma_i}}$ and $\phi^\prime_{K_2^{\Sigma_i}}$ are $0$. Now a direct calculation shows that
\[
(\phi^\prime_{S^{n}}\circ \phi^\prime_{K_1^{\Sigma_i}})_{o^\prime_po_p}\quad\mbox{and}\quad
(\phi^\prime_{K_2^{\Sigma_i}}\circ \phi^\prime_{S^{n}})_{o^\prime_po_p}
\]
are both equal to $(\phi^\prime_{S^{n}})_{o^\prime_po_p}$, and since $(\phi^\prime_{S^{n}})_{o^\prime_po_p}$ is nonzero, this means that the signs of $\phi^\prime_{S^{n}}\circ \phi^\prime_{K_1^{\Sigma_i}}$ and $\phi^\prime_{K_2^{\Sigma_i}}\circ \phi^\prime_{S^{n}}$ have to be the same.

Finally, since Lee homology and sutured Khovanov homology can be obtained from the formal Khovanov bracket by applying additive functors, the same result remains true for the maps $[S^{n}]\circ[K_1^{\Sigma_i}]$ and $[K_2^{\Sigma_i}]\circ [S^{n}]$ and for the induced maps in sutured Khovanov homology.
\end{proof}

\subsection{Direct limits}

Let $K$ be an oriented, framed knot in $A\times I$.  There is a natural Temperley-Lieb cobordism 
$S_k^{[n,n+2]}$ from the $n$-cable $K^{n}$ to the $n+2$-cable $K^{n+2}$, defined by
\[
S_k^{[n,n+2]}:=K^{\cup_{n,n+2}}=\tau(S^1\times\cup_{n,n+2}),
\]
where $\cup_{n,n+2}$ is the cup tangle shown in the left half of Figure~\ref{fig:cup_cap}, and
$\tau$ is the imbedding $\tau\colon S^1\times D^2\times I\rightarrow A\times I\times I$ introduced in Definition~\ref{def:cable_cobordism}.

\begin{figure}
\centerline{
\includegraphics[height=0.6in]{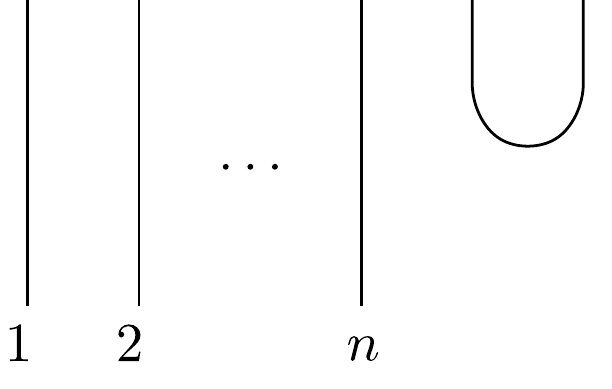}\hspace*{0.8in}
\includegraphics[height=0.6in]{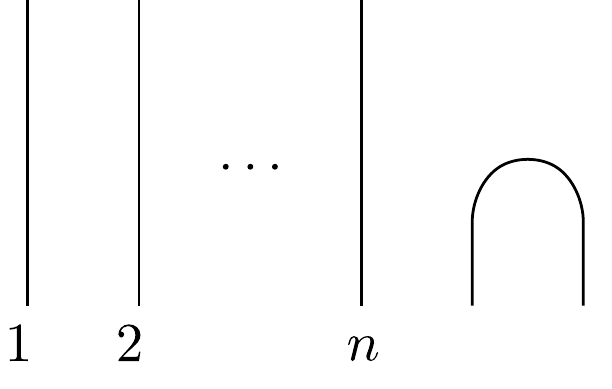}
}
\caption{\label{fig:cup_cap} Tangles $\cup_{n,n+2}$ and $\cap_{n+2,n}$.}
\end{figure}

\begin{lemma}
The map induced by $S_k^{[n,n+2]}$ gives an injection
$$
	\SKh(K^{n}) \hookrightarrow \SKh(K^{n+2}).
$$
\end{lemma}
\begin{proof}
Dually to $S^{[n,n+2]}_k$, we have a cobordism $S^{[n+2,n]}_k:=K^{\cap_{n+2,n}}$, where $\cap_{n+2,n}$ is the tangle shown in the right half of Figure~\ref{fig:cup_cap}. The composition
$$
	\SKh(K^{n}) \rightarrow \SKh(K^{n+2}) \rightarrow \SKh(K^{n})
$$
of the maps induced by $S^{[n,n+2]}_k$ and $S^{[n+2,n]}_k$ is $\pm 2\id$.
Thus the first map $\SKh(K^{n}) \rightarrow \SKh(K^{n+2})$ is injective.
\end{proof}

As a result, we may form the direct limits
$$
	\SKh^{even} (K) = \varinjlim_n \SKh(K^{2n}), 
$$
$$
	\SKh^{odd} (K) = \varinjlim_n \SKh(K^{2n+1}).
$$

These spaces are invariants of the framed knot $K$, and we expect them to have interesting symmetry.  In particular, note that both $\SKh^{even} (K)$ and $\SKh^{odd} (K)$ have commuting actions of $\exsltwo$ and the infinite symmetric group $\fS_\infty = \varinjlim \fS_n$.   Commuting actions of $\sltwo$ and the infinite symmetric group have appeared in the literature recently in connection with the representation theory of infinite-dimensional Lie algebras (see for example
\cite{Vershik} and the references therein).  We therefore pose the following question.

\begin{question}
Can one construct actions of the infinite-dimensional affine Lie algebra $\widehat{\sltwo}$ or the Virasoro algebra on the homology groups $\SKh^{even} (K)$ and 
$\SKh^{odd} (K)$?
\end{question}

\subsection{Colored $\SKh$} Let $K$ be an oriented, framed knot in $A\times I$. We define the $n$-colored sutured Khovanov homology of $K$ as the subspace
\[
\SKh_n(K):=\SKh(K^{n})^{\fS_n}\subset\SKh(K^{n})
\]
of $\fS_n$ invariants inside $\SKh(K^{n})$. This definition is motivated by the following result, which holds for ordinary Khovanov homology of $n$-cables of knots in $\R^3$:

\begin{theorem}[\cite{invariant_subspace}] Let $K$ be an oriented, framed knot in $\R^3$. Then the subspace of $\fS_n$ invariants inside the Khovanov homology of $K^{n}$ (with coefficients in a field of characteristic $0$) is isomorphic to Khovanov's categorification of the non-reduced $n$-colored Jones polynomial of $K$ \cite{MR2124557}.
\end{theorem}


\section{The category of finite-dimensional representations of $\exsltwo$}\label{sec:repexsltwo}
Let $\mbox{rep}(\exsltwo)$ denote the category of finite-dimensional graded representations of $\exsltwo$.  In this section we give a quiver description of the category $\mbox{rep}(\sltwo)$, which is seen to be equivalent to the category of finitely-generated graded representations of of a finite-dimensional Koszul algebra.  We also give a quiver description of the category of finite-dimensional graded representations of $\exsltwo_{dg}$.

Let $\Gamma$ denote the quiver with vertex set $\N = \{0,1,2,\dots\}$ and a single oriented edge from vertex $i$ to each of the vertices $i-2,i$ and $i+2$.  (Thus the underlying graph of $\Gamma$ has two connected components, one of which contains the odd vertices and the other of which contains the even vertices.)  We denote the individual edges of $\Gamma$ by $\alpha_i$, $\beta_i$, and $\epsilon_i$ as in the picture below. 

$$
\begin{tikzpicture}[->,>=stealth',shorten >=1pt,auto,node distance=3cm,
  thick,main node/.style={circle,fill=blue!20,draw,font=\sffamily\Large\bfseries}]

  \node[main node] (0) {0};
  \node[main node] (2) [right of=0] {2};
  \node[main node] (4) [right of=2] {4};
  \node[main node] (6) [right of=4] {6};
  
   \node[main node] (1) [below right of=0]{1};
  \node[main node] (3) [right of=1] {3};
  \node[main node] (5) [right of=3] {5};

  \path[every node/.style={font=\sffamily\small}]
     (0) edge [loop above] node {$\epsilon_0$} (0)

     (2) edge [loop above] node {$\epsilon_2$} (2)

    (4) edge [loop above] node {$\epsilon_4$} (4)
    
    (6) edge [loop above] node {$\epsilon_6$} (6)
    
    (0) edge [bend right] node {$\alpha_0$} (2)
    
     (2) edge [bend right] node {$\alpha_2$} (4)
     
     (4) edge [bend right] node {$\alpha_4$} (6)

    (2) edge [bend right] node {$\beta_0$} (0)
    
      (4) edge [bend right] node {$\beta_2$} (2)
      
      (6) edge [bend right] node {$\beta_4$} (4)

       (1) edge [loop above] node {$\epsilon_1$} (1)

     (3) edge [loop above] node {$\epsilon_3$} (3)
     
     (5) edge [loop above] node {$\epsilon_5$} (5)
   
 (1) edge [bend right] node {$\alpha_1$} (3)
    
     (3) edge [bend right] node {$\alpha_3$} (5)

         (3) edge [bend right] node {$\beta_1$} (1)
    
      (5) edge [bend right] node {$\beta_3$} (3)
      
      ;
      \end{tikzpicture} \dots
$$

\begin{proposition}\label{prop:quiver1}
The category $\mbox{rep}(\exsltwo)$ is equivalent to the category of representations of the quiver $\Gamma$ with the following relations.  For all $i\in \N$ we have
\begin{itemize}
\item $\alpha_{i+2}\alpha_i = 0$;
\item $\beta_{i-2}\beta_i = 0$;
\item $\epsilon_{i+2}\alpha_i + \alpha_i\epsilon_i=0$;
\item $\epsilon_{i}\beta_{i} + \beta_i\epsilon_{i+2}=0$;
\item $\beta_i\alpha_i + \alpha_{i-2}\beta_{i-2} + \epsilon_i^2 = 0$; and
\item $\beta_i\alpha_i + \frac{i^2}{4(i+3)} \epsilon_i^2=0$.
\end{itemize}
\end{proposition}
By convention, we take $\alpha_i=\beta_i=\epsilon_i = 0$ in the above relations when $i<0$.  


\begin{proof}
We describe the functor $Q$ from $\exsltwo$ representations to $\Gamma$ representations which gives the equivalence.  Let $M$ be a finite-dimensional representation of $\exsltwo$, and for $i\in \N$, let $E_i = \{m\in M : h(m) = i m\text{ and } e(m) = 0\}$ denote the space of highest weight vectors of $M$ regarded as a finite-dimensional representation of $\sltwo$.  Then the commutation relations between $e$ and $v_j$, $j=2,0,-2$, give
\begin{itemize}
\item $v_2 : E_i \rightarrow E_{i+2}$,
\item $v_0 : E_i \rightarrow E_i\oplus f(E_{i+2})$, and
\item $v_{-2} :  E_i \rightarrow E_{i-2}\oplus f(E_{i}) \oplus f^2(E_{i+2})$.
\end{itemize}

Moreover, for $m\in E_i$ we may write
$$
	v_0(m) = p_i(m) - \frac{2}{i+2}fv_2(m),
$$


and
$$
	v_{-2}(m) = q_{i-2}(m) + \frac{1}{i}fp_i(m) - \frac{1}{(i+1)(i+2)}f^2v_2(m),
$$
where $p_i(m) \in E_i$ and $q_{i-2}(m)\in E_{i-2}$.


Now, setting 
$$
	\alpha_i = (i+3)v_2: E_i \longrightarrow E_{i+2},
$$
$$
	\epsilon_i = \frac{1}{i}p_i: E_i \longrightarrow E_i, \text{ and }
$$
$$
	\beta_i = \frac{1}{i+2} q_i : E_{i+2} \longrightarrow E_{i},
$$
the defining relations $[v_k,v_l]=0$ for $k,l\in \{2,0,-2\}$ give rise to the relations given in the theorem.
The inverse functor $Q^{-1}$ takes a representation $E$ of the quiver $\Gamma$ and declares the vector space $E_i$ to be the space of highest weight vectors of weight $i$; this specifies the action of $e,f,h$ on $Q^{-1}(E)$.  The action of $v_2,v_0,v_{-2}$ on highest weight vectors is then determined by the representation of the quiver, together with
$$
	v_0(m) = p_i(m) - \frac{2}{i+2}fv_2(m),
$$
$$
	v_{-2}(m) = q_{i-2}(m) + \frac{1}{i}fp_i(m) - \frac{1}{(i+1)(i+2)}f^2v_2(m),
$$
and the commutation relations between $\sltwo$ and $v_{2},v_0,v_{-2}$ determine the action on the rest of the representation.  It is then clear that $Q$ and $Q^{-1}$ are inverse equivalences.
\end{proof}

\begin{remark} An analog of the theorem above for the (non super) current Lie algebra $\sltwo(V_{(1)})$, where $V_{(1)}$ is the 2-dimensional irrep of $\sltwo$, is due to Loupias \cite{Loupias}; see also \cite{HK}.
\end{remark}

\section{Examples}\label{sec:ex}

\subsection{Schur-Weyl representation and trivial braid closures} \label{sec:SchurWeylRep}

Recall that if $V = \C^2$ is the defining representation of $\mathfrak{sl}_2$, then we have a natural action of $\fS_n$ on the $n$--fold tensor product, extended linearly from:

\[\sigma(v_1 \otimes \ldots \otimes v_n) := v_{\sigma^{-1}(1)} \otimes \ldots \otimes v_{\sigma^{-1}(n)}\] for $\sigma \in \fS_n$, $v_i \in V$. We also have the induced tensor product action of $\mathfrak{sl}_2$, extended  $\C$--linearly from:

\[x(v_1 \otimes \ldots \otimes v_n) := \sum_{i=1}^n v_1 \otimes \ldots \otimes x(v_i) \otimes \ldots \otimes v_n\] for $x \in \mathfrak{sl}_2$. These actions commute. We will refer to the resulting action of $\mathfrak{sl}_2 \times \fS_n$ on $V^{\otimes n}$ as the {\em Schur-Weyl representation}.

\begin{remark} \label{rmk:sign}
Note that strictly speaking, in what follows we will be considering the Schur-Weyl representation on $V^{\otimes \lceil\frac{n}{2}\rceil} \otimes (V^*)^{\otimes \lfloor\frac{n}{2}\rfloor}$ (where the order of the terms in the tensor product alternates between $V$ and $V^*$), using the isomorphism $\phi: V \rightarrow V^*$ determined by $\phi(v_{\pm}) = \pm \bar{v}_{\pm}$. As a consequence, the action of the transposition $(i \,\, j) \in \fS_n$ on a tensor product of standard basis vectors will carry the sign $(-1)^{i-j}$.
\end{remark}

\begin{proposition} Let $\Id_n$ denote the trivial $n$--strand braid and $\widehat{\Id}_n$ its closure, understood as the $0$--framed $n$--cable of the unknot.
\begin{enumerate}
	\item The actions of $v_{-2}, v_0, v_2 \in \exsltwo$ on $\SKh(\widehat{\Id}_n)$ are trivial, hence the action of $\exsltwo$ on $\SKh(\widehat{\Id}_n)$ reduces to an action of $\sltwo$.
	\item The commuting actions of $\sltwo$ and $\fS_n$ on $\SKh(\widehat{\Id}_n)$ agree with the Schur-Weyl representation.
\end{enumerate}
\end{proposition}

\begin{proof}
Using the functor $\cF$ described in Section \ref{subs:annulartqft} applied to the crossingless diagram for $\widehat{\Id}_n$, we see that as an $\sltwo$--representation, $\SKh(\widehat{\Id}_n) \cong V^{\otimes \lceil{\frac{n}{2}}\rceil} \otimes (V^*)^{\otimes \lfloor{\frac{n}{2}\rfloor}} \cong V^{\otimes n}$ and is concentrated in $(i,j')$--grading $(0,0)$. It follows that the actions of $v_{-2}, v_0, v_2 \in \exsltwo$ are trivial, as each shifts the $i$ grading by $1$.

We would now like to see that the action of $\fS_n$ on $\SKh(\widehat{\Id}_n)$ agrees with the standard commuting action of $\fS_n$ in the Schur-Weyl representation.  If we regard a standard basis element of $\SKh(\widehat{\Id}_n)$--i.e., one of the form \[v_{\pm} \otimes \bar{v}_{\pm} \otimes \ldots \bar{v}_{\pm} \otimes v_{\pm} \in V \otimes V^* \otimes \ldots \otimes V^* \otimes V\] (in the odd $n$ case)--as a labeling of the corresponding circles of the (unique) resolution by $+$'s and $-$'s, this amounts to verifying that the transposition $t_i = (i \,\, i+1) \in \fS_n$ exchanges the markings on the $i$th and $(i+1)$st strands and multiplies the resulting vector by $-1$. This follows by appealing to Proposition \ref{prop:kbrelation}. In particular, the cobordism map associated to $t_i$ is $\mbox{id}+u_i$, where $u_i$ is the Temperley-Lieb map described by the cobordism in Figure \ref{fig:TLCob}  (using Conventions \ref{conv:one} and \ref{conv:two} to pin down signs). One quickly computes that the map $u_i$ is $0$ on any standard basis vector whose $i$th and $(i+1)$st labels agree. If the $i$th and $(i+1)$st labels disagree, one computes
\[u_i(\ldots \otimes (v_{\pm} \otimes \bar{v}_{\mp})\otimes \ldots) = (\ldots \otimes (-v_{\pm} \otimes \bar{v}_{\mp} - v_{\mp} \otimes \bar{v}_{\pm}) \otimes \ldots).\] We conclude that the action of $t_i = 1+u_i$ agrees with the $\fS_n$ action in the Schur-Weyl representation, as desired.
\end{proof}

\subsection{Positive stabilizations of the nontrivial annular unknot}
Let
\begin{align*}
\widehat{\beta_n}\quad&:=\quad\mbox{the $n$-fold positive stabilization of the nontrivial unknot}\\
&\;=\quad\mbox{the annular closure of the braid $\beta_n:=\sigma_1\sigma_2\ldots\sigma_n\in\mathfrak{B}_{n+1}$,}\\[0.03in]
V_{(m)}\quad&:=\quad\mbox{the $(m+1)$-dimensional irreducible representation of $\sltwo$}.
\end{align*}

\begin{proposition}\label{prop:posstab} For all $n\geq 0$, we have
\[
\SKh^i(\widehat{\beta_n})\cong\begin{cases}
V_{(n+1)}\{n\}&\mbox{if $i=0$,}\\
V_{(n-1)}\{n+2\}&\mbox{if $i=1$,}\\
0&\mbox{else,}
\end{cases}
\]
where $\{m\}$ denotes the grading-shift functor which raises the $j':=(j-k)$-degree by $m\in\Z$ and preserves the $k$-degree.
As module over $\exsltwo$, $\SKh(\widehat{\beta_n})$ is indecomposable, with module structure determined by the $\sltwo$ decomposition above together with the fact that the generator $v_{-2}$ of $\exsltwo$ takes a highest weight vector in $V_{(n+1)}$ to a highest weight vector in $V_{(n-1)}$.

\end{proposition}
\begin{proof} The proof goes by induction on $n$. For $n=0$, we have
\[
\SKh^i(\widehat{\beta_0})=\SKh^i(\mbox{nontrivial annular unknot})=\begin{cases}
V_{(1)}\{0\}&\mbox{if $i=0$,}\\
0&\mbox{else,}
\end{cases}
\]
and hence the statement of the proposition is satisfied because $V_{(-1)}=0$. For $n=1$, the sutured annular Bar-Natan complex of $\widehat{\beta_n}$ is isomorphic to
\[
0\quad\longrightarrow\quad
\begin{matrix}
V_{(2)}\{1\}\\
\oplus\\
V_{(0)}\{1\}
\end{matrix}\quad
\xrightarrow{\delta_0}\quad
\begin{matrix}
V_{(0)}\{3\}\\
\oplus\\
V_{(0)}\{1\}
\end{matrix}\quad
\longrightarrow\quad 0
\]
where
\[
\delta_0=\begin{bmatrix}
0&0\\
0&1
\end{bmatrix},
\]
and so the proposition holds in this case as well.

To prove the proposition for $n>1$, we use that the sutured annular Khovanov complex of $\widehat{\beta_n}$ -- denoted $C(\widehat{\beta_n})$ -- can be written as a mapping cone
\[
C(\widehat{\beta_n})\cong\operatorname{Cone}\left(C(\widehat{\beta_{n;0}})\{1\} \xrightarrow{\;\; f\;\;} C(\widehat{\beta_{n;1}})\{2\}\right),\footnote{
The diagram $\widehat{\beta_{n;0}}$ does not inherit a consistent orientation from $\widehat{\beta_n}$. However, it turns out that because of the particular form of $\widehat{\beta_n}$, one can choose an orientation for $\widehat{\beta_{n;0}}$ which almost agrees with the orientation of $\widehat{\beta_n}$, in the sense that it differs from the latter orientation only along a crossingless arc. The grading shifts in the mapping cone description of $C(\widehat{\beta_n})$ arise because, in the construction of Khovanov homology, the $j$-degree is shifted by $r+n_+-2n_-$, where $r$ denotes the number of $1$-resolutions, and $n_+$/$n_-$ denotes the number of positive/negative crossings.}
\]
where $\widehat{\beta_{n;0}}$ (resp., $\widehat{\beta_{n;1}}$) denotes the annular link diagram obtained from $\widehat{\beta_n}=\widehat{\sigma_1\ldots\sigma_n}$ by replacing the unique crossing in $\sigma_n$ by its $0$-resolution (resp., $1$-resolution), and $f$ is the chain map induced by a saddle cobordism between the $0$- and the $1$-resolution of this crossing. It follows from the properties of the mapping cone that there is a short exact sequence of chain complexes
\[
0\longrightarrow C^{*-1}(\widehat{\beta_{n;1}})\{2\}\longrightarrow C^*(\widehat{\beta_n})\longrightarrow C^*(\widehat{\beta_{n;0}})\{1\}\longrightarrow 0,
\]
which induces a long exact sequence in homology:
\[
\ldots\longrightarrow \SKh^{i-1}(\widehat{\beta_{n;1}})\{2\} \longrightarrow \SKh^i(\widehat{\beta_n}) \longrightarrow \SKh^i(\widehat{\beta_{n;0}})\{1\} \longrightarrow \SKh^{i}(\widehat{\beta_{n;1}})\{2\}\longrightarrow\ldots.
\]

Looking at $\widehat{\beta_{n;0}}$ and $\widehat{\beta_{n;1}}$, one further sees that these diagrams represent the same annular links as the diagrams $\widehat{\beta_{n-1}\times 1}$ and $\widehat{\beta_{n-2}}$, respectively, and hence one can write the above long exact sequence as
\[
\ldots\longrightarrow \SKh^{i-1}(\widehat{\beta_{n-2}})\{2\} \longrightarrow \SKh^i(\widehat{\beta_n}) \longrightarrow \SKh^i(\widehat{\beta_{n-1}})\otimes V_{(1)}\{1\} \longrightarrow \ldots,
\]
where we have used that $\SKh^i(\widehat{\beta_{n-1}\times 1})= \SKh^i(\widehat{\beta_{n-1}})\otimes V_{(1)}$.\footnote{This equation certainly holds as an equation between graded vector spaces. Since the $\sltwo$-module structure is determined up to isomorphism by the $k$-grading, the equation also holds as an equation between $\sltwo$-modules.}

We now use induction on $n$ and the fact that $V_{(m)}\otimes V_{(1)}\cong V_{(m+1)}\oplus V_{(m-1)}$ to write the nontrivial part of the above long exact sequence as
\[\xymatrix@R=.4in@C=0.55in{0\ar[r] &\SKh^0(\widehat{\beta_n})\ar[r]&V_{(n+1)}\{n\}\oplus V_{(n-1)}\{n\}\ar[dll]_{c_0}\\
V_{(n-1)}\{n\}\ar[r]&\SKh^1(\widehat{\beta_n})\ar[r]&V_{(n-1)}\{n+2\}\oplus V_{(n-3)}\{n+2\}\ar[dll]_{c_1}\\
V_{(n-3)}\{n+2\}\ar[r]&\SKh^2(\widehat{\beta_n})\ar[r]&0,}
\]
where $c_0$ and $c_1$ are connecting homomorphisms.

\begin{lemma}\label{lem:connecting_posstab}
$c_0$ and $c_1$ are nonzero.
\end{lemma}
\begin{proof} Let $g_{n-1}\in C^0(\widehat{\beta_{n-1}})\{1\}$ denote the elment obtained by labeling all circles in the all-$0$-resolution of $\widehat{\beta_{n-1}}$ by $v_-$. Moreover, let $R_{n-1}(\ell)$ denote the resolution of $\widehat{\beta_{n-1}}$ obtained by choosing the $1$-resolution at the $\ell$th crossing of $\widehat{\beta_{n-1}}$ and the $0$-resolution at all other crossings. Further, let $g'_{n-1}(\ell)\in C^1(\widehat{\beta_{n-1}})\{1\}$ be the element given by labeling the trivial circle in $R_{n-1}(\ell)$ by $w_+$ and each nontrivial circle in $R_{n-1}(\ell)$ by $v_-$. Define
\[
g'_{n-1}:=\bigoplus_{\ell=1}^{n-1} g'_{n-1}(\ell)\quad\in\quad C^1(\widehat{\beta_{n-1}})\{1\}.
\]
We now leave it to the reader to verify that
\[
c_0([g_{n-1}\otimes v_+])=[g_{n-2}]\quad\mbox{and}\quad c_1([g'_{n-1}\otimes v_+])=[g'_{n-2}].
\]

Using the same sign conventions as in Bar-Natan's first paper on Khovanov homology, one can further see that the elements $g_{n-1},g'_{n-1},g_{n-2},g'_{n-2}$ are cycles, and that none of them is a boundary. It follows that $c_0$ and $c_1$ are nonzero.
\end{proof}

The inductive step in the proof of Proposition~\ref{prop:posstab} now follows from the above long exact sequence and from Lemma~\ref{lem:connecting_posstab}, coupled with the facts that (a) an $\sltwo$-module map between two non-isomorphic irreducible $\sltwo$-modules is necessarily zero, and (b) an $\sltwo$-module map between two isomorphic irreducible $\sltwo$-modules is either zero or an isomorphism.

The claim about the action of $\exsltwo$ is now a straightforward computation which we leave as an exercise.

\end{proof}

\subsection{Annular $(2,-n)$-torus links for $n\geq 0$}
Let
\begin{align*}
T_{2,-n}\quad&:=\quad\mbox{the annular $(2,-n)$-torus link}\\
&\;=\quad\mbox{the annular closure of the braid $\sigma_1^{-n}\in\mathfrak{B}_{2}$.}
\end{align*}

In the case where $n$ is even, we assume that both components of $T_{2,-n}$ are oriented parallel to each other, in direction of the braid $\sigma_1^{-n}$.

\begin{proposition}\label{prop:torus} For all $n\geq 1$, we have
\[
\SKh^i(T_{2,-n})\cong\begin{cases}
V_{(2)}\{-n\}&\mbox{if $i=0$,}\\
V_{(0)}\{2i-n\}&\mbox{if $-n\leq i\leq -1$ and $i$ odd,}\\
V_{(0)}\{2i+2-n\}&\mbox{if  $-n+1\leq i\leq-2$ and $i$ even,}\\
V_{(0)}\{-3n+2\}\oplus V_{(0)}\{-3n\}&\mbox{if $i=-n$ and $n$ even,}\\
0&\mbox{else.}
\end{cases}
\]
The $\exsltwo$ module structure on $\SKh(T_{2,-n})$ is completely determined by the fact that the generator $v_{2}$ of $\exsltwo$ takes a highest weight vector of the summand $V_{(0)}\{-n-2\}$ to a highest weight vector of $V_{(2)}\{-n\}$ and annihilates all other $V_{(0)}$ summands..  Thus $\SKh(T_{2,-n})$ is an indecomposable $\exsltwo$ module if and only if $n=1$.  
\end{proposition}
\begin{proof} The proof goes by induction on $n$ and is similar to the proof of Proposition~\ref{prop:posstab}. For $n=1$, the complex $C(T_{2,-n})$ is isomorphic to
\[
0\quad\longrightarrow\quad
\begin{matrix}
V_{(0)}\{-1\}\\
\oplus\\
V_{(0)}\{-3\}
\end{matrix}\quad
\xrightarrow{\delta_{-1}}\quad
\begin{matrix}
V_{(2)}\{-1\}\\
\oplus\\
V_{(0)}\{-1\}
\end{matrix}\quad
\longrightarrow\quad 0
\]
where
\[
\delta_0=\begin{bmatrix}
0&0\\
1&0
\end{bmatrix},
\]
and hence the proposition is satisfied in this case.

To prove the proposition for $n>1$, we write $C(T_{2,-n})$ as a mapping cone
\[
C(T_{2,-n})\cong\operatorname{Cone}\left(C(T_{2,-n;0})\{-3n+1\}[-n] \xrightarrow{\;\; g\;\;} C(T_{2,-n;1})\{-1\}[-1]\right),
\]
where $[m]$ denotes a shift of the homological grading by $m\in\Z$, and $T_{2,-n;0}$ (resp., $T_{2,-n;1}$) denotes the diagram obtained from $T_{2,-n}=\widehat{\sigma_1^{-n}}$ by replacing the crossing in the last $\sigma_1^{-1}$ in $\sigma_1^{-n}$ by its $0$-resolution (resp., its $1$-resolution).\footnote{The diagram $T_{2,-n;0}$ does not inherit a consistent orientation from $T_{2,-n}$. One therefore has to choose an orientation for $T_{2,-n;0}$ ``by hand''.} As in the proof of Proposition~\ref{prop:posstab}, we obtain a long exact sequence
\[
\ldots\longrightarrow\SKh^{i}(T_{2,-n;1})\{-1\}\longrightarrow\SKh^i(T_{2,-n})\longrightarrow\SKh^{i+n}(T_{2,-n;0})\{-3n+1\}\longrightarrow\ldots,
\]
and by observing that  $T_{2,-n;0}$ represents a trivial annular unknot, and $T_{2,-n;1}$ represents $T_{2,-(n-1)}$, we can write this long exact sequence as
\[
\ldots\rightarrow\SKh^{i}(T_{2,-(n-1)})\{-1\}\rightarrow\SKh^i(T_{2,-n})\rightarrow\SKh^{i+n}(\mbox{trivial unknot})\{-3n+1\}\stackrel{c_{i+n}}{\rightarrow}\ldots,
\]
where
\[
c_{i+n}\colon\SKh^{i+n}(\mbox{trivial unknot})\{-3n+1\}\longrightarrow\SKh^{i+1}(T_{2,-(n-1)})\{-1\}
\]
is the connecting homomorphism.

\begin{lemma}\label{lem:connecting_torus} $c_{i+n}$ is zero unless $n$ is odd and $i+n=0$. Moreover, if $n$ is odd, then
\[
c_0\colon \SKh^0(\mbox{trivial unknot})\{-3n+1\}\longrightarrow\SKh^{-n+1}(T_{2,-(n-1)})\{-1\}
\]
is conjugate to the map
\[
\begin{matrix}
V_{(0)}\{-3n+2\}\\
\oplus\\
V_{(0)}\{-3n\}
\end{matrix}\quad\xrightarrow{c'_0}\quad
\begin{matrix}
V_{(0)}\{-3n+4\}\\
\oplus\\
V_{(0)}\{-3n+2\}
\end{matrix}
\]
given by
\[
c'_0=\begin{bmatrix}
0&0\\
1&0
\end{bmatrix},
\]
where we have used that the graded $\sltwo$-module $\SKh^{-n+1}(T_{2,-(n-1)})\{-1\}$ is isomorphic to $V_{(0)}\{-3n+4\}\oplus V_{(0)}\{-3n+2\}$ by induction.
\end{lemma}

To prove Lemma~\ref{lem:connecting_torus}, we need the following claim:

\begin{claim} For $n>1$, the isomorphism
\[
\phi\colon\SKh(\mbox{trivial unknot})\longrightarrow\SKh(T_{2,-n;0})
\]
induced by a sequence of $n-1$ consecutive Reidemeister I moves is given by
\begin{align*}
\phi(w_+)&=\sum_{\ell=1}^n(-1)^{\ell+1}w_-^{\otimes(\ell-1)}\otimes w_+\otimes w_-^{\otimes (n-\ell)},\\
\phi(w_-)&=w_-^{\otimes n},
\end{align*}
where the terms on the right-hand side live in the vector space associated to the all-$0$-resolution of $T_{2,-n;0}$. (Here we assume the order of the tensor factors corresponds to the order in which the circles of the all-$0$-resolution appear as one travels around the annulus).
\end{claim}

The proof of the claim is an easy computation and therefore omitted.

\begin{proof}[Proof of Lemma~\ref{lem:connecting_torus}]
Since $\SKh(\mbox{trivial unknot})$ is supported in homological degree $0$, it is clear that $c_{i+n}$ is zero unless $i+n=0$. By observing that $c_0$ is induced by the chain map
\[
g\colon C(T_{2,-n;0})\{-3n+1\}[-n] \longrightarrow C(T_{2,-n;1})\{-1\}[-1]
\]
that appears in the mapping cone description of $C(T_{2,-n})$, it is further easy to see that $c_0$ can be written as \[c_0=(M\circ\phi)_*,\] where $\phi$ is as in the claim, and $M$ is the map
\[
C^0(T_{2,-n;0})\{-3n+1\}=W^{\otimes n}\{-2n\}\xrightarrow{\;\; M\;\;} W^{\otimes(n-1)}\{-2n+1\}=C^{-n+1}(T_{2,-(n-1)})\{-1\}
\]
given by $M(w_1\otimes w_2\otimes\ldots w_{n-1}\otimes w_n):=m(w_1\otimes w_n)\otimes w_2\otimes\ldots\otimes w_{n-1}$, with $m$ denoting Khovanov's multiplication map. We thus obtain $c_0([w_-])=[M(\phi(w_-))]=0$ and
\[
c_0([w_+])=[M(\phi(w_+))]=\begin{cases} 0&\mbox{if $n$ is even,}\\
[2w_-^{\otimes (n-1)}]&\mbox{if $n$ is odd.}
\end{cases}
\]

Now observe that $2w_-^{\otimes (n-1)}\in C^{-n+1}(T_{2,-(n-1)})\{-1\}$ cannot be a boundary because it sits in lowest possible homological degree. Hence $c_0([w_+])$ is a nonzero whenever $n$ is odd, and by looking at the gradings, one can see that $c_0([w_+])$ lives in $V_{(0)}\{-3n+2\}\subset\SKh^{-n+1}(T_{2,-(n-1)})\{-1\}$. It is now evident that $c_0$ has the desired form.
\end{proof}

The inductive step in the proof of Proposition~\ref{prop:torus} now follows from Lemma~\ref{lem:connecting_torus} and from the long exact sequence stated before Lemma~\ref{lem:connecting_torus}.

The action of $v_{-2}$ is now an easy computation which is left to the reader.
\end{proof}

\begin{remark} Comparing Proposition~\ref{prop:torus} to the computations in Section 6 of \cite{K}, we see that
the ranks of $\SKh(T_{2,-n})$ and $\Kh(T_{2,-n})$ agree in all homological degrees except degrees 0 and 1. 
\end{remark}

\section{Appendix}\label{sec:ap}
In the following, let ${\bf z} := \{(0,0,z) \, \vline \, z \in \R\} \subset \R^3$ denote the $z$-axis in $\R^3$, and $I = [0,1]$.  

\begin{definition}
An (oriented) {\em annular link cobordism} $\Sigma$ between (oriented) links $L_0 \subset \R^3 \setminus {\bf z} = (\R^3 \setminus {\bf z}) \times \{0\}$ and $L_1 \subset (\R^3 \setminus {\bf z}) = (\R^3 \setminus {\bf z}) \times \{1\}$ is a smooth, compact, (oriented) surface imbedded in $(\R^3 \setminus {\bf z}) \times I$ satisfying 
	\begin{itemize}
		\item $\partial\Sigma = L_0 \amalg L_1$ and
		\item there exists some $\epsilon > 0$ such that the intersection of $\Sigma$ with$(\R^3 \setminus {\bf z}) \times \left([0,\epsilon] \amalg [1-\epsilon,\epsilon]\right)$ can be identified with the product imbedding $(L_0 \times [0,\epsilon]) \amalg (L_1 \times [1-\epsilon, 1])$.
	\end{itemize}

\end{definition}

An annular link cobordism $\Sigma$ is said to be {\em generic} if the projection map $p: (\R^3 \setminus {\bf z}) \times I \rightarrow I$ restricted to $\Sigma$ is Morse with distinct critical values.

\begin{definition} An {\em annular movie} of a link cobordism is a smooth, one-parameter family of curves, $D_t \subset (\R^2 \setminus {\bf 0})$, $t \in [0,1]$, called {\em annular stills}, satisfying:
	\begin{itemize}
		\item for all but finitely many $t \in [0,1]$, $D_t$ is a link diagram (an immersed curve equipped with over/undercrossing information at all double points),
		\item at each of the finitely many {\em critical levels} $t_1, \ldots, t_k$, the diagram undergoes a single {\em elementary string interaction} (i.e., a birth, saddle, death, or Reidemeister move) localized to a disk in $\R^2 \setminus {\bf 0}$.
	\end{itemize}
\end{definition}

\begin{remark} Since annular cobordisms are assumed compact, we can (and shall) consider all annular cobordisms to be imbedded in $(A \times I) \times I  \subset (\R^3 \setminus {\bf z}) \times I$. Accordingly, an annular movie may be viewed upon $A \subset \R^2 \setminus {\bf 0}$.
\end{remark}

\begin{lemma} Any annular link cobordism $\Sigma$ can be represented by an annular movie.
\end{lemma}

\begin{proof} Let $\Sigma$ be an annular link cobordism. Composing with the inclusion $(\R^3 \setminus {\bf z}) \rightarrow \R^3$ produces a traditional link cobordism (cf. \cite[Defn. 5]{Jacobsson}) which can be perturbed in a small open neighborhood (hence, in the complement of ${\bf z} \times I$) to a generic (annular) link cobordism. The image of $\Sigma$ under the projection map, \[\pi \times \mbox{id}: \R^3_{(x,y,z)} \times I \rightarrow \R^2_{(x,y)} \times I,\] is then an {\em annular broken surface diagram}, an immersed surface in $(\R^2 \setminus {\bf 0}) \times I$ whose points of self-intersection are either generic double points, triple points, or branch points (cf. \cite[Sec. 1.4]{CarterKamadaSaito}). After a possible further perturbation of $\Sigma$ (which can, again, be performed in the complement of ${\bf z} \times I$), the intersections of the annular broken surface diagram with the level sets $(\R^2 \setminus {\bf 0}) \times \{t\}$ yield an annular movie of the link cobordism, as desired.
\end{proof}

\begin{lemma} Let $\Sigma$ be an annular link cobordism represented by two different annular movies ${\bf M}_\Sigma$ and ${\bf M}'_\Sigma$. Then ${\bf M}_\Sigma$ and ${\bf M}'_\Sigma$ are related by a finite sequence of Carter-Saito {\em movie moves} (\cite[Fig. 23-37]{carter_saito}), each of which is localized to a disk in $\R^2 \setminus {\bf 0}$.
\end{lemma}

\begin{proof} As before, $\Sigma \subseteq (\R^3 \setminus {\bf z}) \times I \subseteq \R^3 \times I$ can be viewed as a traditional link cobordism, and the isotopy joining the representatives of $\Sigma$ giving rise to ${\bf M}_\Sigma$ and ${\bf M}'_\Sigma$, respectively, can be viewed as an isotopy in $\R^3 \times I$. Carter-Saito's movie move theorem \cite[Thm. 7.1]{carter_saito} 
 then implies that there exists some finite sequence of movie moves, each localized to a disk in $\R^2$, relating {\bf M} and {\bf M}'.

We claim that each movie move can, in fact, be localized to a disk in $\R^2 \setminus {\bf 0}$. But this follows because if any of the movie moves is localized to a disk which cannot be made disjoint from ${\bf 0}$, then in the course of the movie move, there exists some still whose curve intersects ${\bf 0}$. The corresponding isotopy it represents must therefore intersect ${\bf z} \times I$, a contradiction.
\end{proof}
\bibliography{annular}

\end{document}